\documentclass[12pt]{article}

\usepackage{amsmath,amssymb,amsfonts,longtable,amscd,amsthm}
\usepackage[all,cmtip]{xy}
\usepackage[caption=false]{subfig}
\usepackage{graphicx}
\usepackage{array}
\usepackage{multirow}
\usepackage[bookmarksnumbered,bookmarksopen,pdfusetitle,unicode]{hyperref}
\usepackage{color}

\newtheorem{theorem}{Theorem}[section]
\newtheorem{corollary}[theorem]{Corollary}
\newtheorem{conj}[theorem]{Conjecture}
\newtheorem{lemma}[theorem]{Lemma}
{
\newtheorem{prop}[theorem]{Proposition}
}
\theoremstyle{definition}

\newtheorem{defn}[theorem]{Definition}
\newtheorem{remark}[theorem]{Remark}
\newtheorem{example}[theorem]{Example}

\renewcommand\footnotemark{}

\newenvironment{enum}
 {\begin{list}{\labelitemi}{\leftmargin=0em \itemindent=0.5em}}
 {\end{list}}

\newcommand{\rar}{rational tree with almost reduced Artin cycle}

\newcommand{\Addresses}{{
  \bigskip
  \small

  Kazunori Nakamoto, \textsc{{    Center for medical educations and sciences, 
  Faculty of Medicine, University of Yamanashi, Yamanashi, Japan}}\par\nopagebreak
  \textit{e-mail address}: \texttt{nakamoto@yamanashi.ac.jp}

  \medskip

  Ay{\c s}e Sharland, \textsc{Department of Mathematics, University of Rhode Island, Kingston, RI 02881, US}
  \par\nopagebreak
  \textit{e-mail address}: \texttt{aysharland@uri.edu}

  \medskip

Meral Tosun,  \textsc{Department of Mathematics, Galatasaray University, Ortak{\"o}y 34357, Istanbul, Turkey}\par\nopagebreak
  \textit{e-mail address}: \texttt{mrltosun@gmail.com}

}}

\title{
 {Triple Root Systems, Rational Quivers and Examples of Linear Free Divisors}} 

\author{K. Nakamoto, A. Sharland and M. Tosun}

\begin{document}
\maketitle

\begin{flushright}\begin{small}Dedicated to Jose Seade on his 60th birthday. \end{small}\end{flushright}
\let\thefootnote\relax

\begin{abstract} The dual resolution graphs of rational triple point singularities  can be seen as a generalisation of Dynkin diagrams. In this work, we study the relation between the root systems corresponding to those diagrams. We determine the number of roots for each rational triple point singularity, and show that for each root we obtain a linear free divisor. Furthermore, we deduce that linear free divisors defined by {    rational triple quivers with roots in the corresponding triple root systems} satisfy the global logarithmic comparison theorem. We also discuss {   a} generalisation of these results to the class of rational singularities with almost reduced Artin cycle. 
\end{abstract}

\section{Introduction}

Let $D$ be a reduced hypersurface in $\mathbb{C}^N$.  In \cite{saito}, K. Saito associated $D$ with the sheaf of logarithmic vector fields along $D$, denoted by 
$\textnormal{Der}(-\textnormal{log }D)$, which is 
the subsheaf of the holomorphic vector fields on $(\mathbb C^N,0)$.  If $\textnormal{Der}(-\textnormal{log }D)$ is a locally free $\mathcal O_ {\mathbb{C}^N,0}$-module then $D$ is said to be a free divisor (also known as Saito divisor in the literature). The basic examples of a free divisor are normal crossing divisors and reduced plane curve singularities (\cite{saito}). The example motivating K. Saito and attracting the attention of many other mathematicians on the subject is the freeness of discriminants in the base space of the versal deformation of isolated hypersurface singularities. In \cite{buch-mond}, Buchweitz and Mond introduced a new class of free divisors arising from the representation of quivers. The discriminant in the representation space of a quiver with {   a} given dimension vector is a \textit{linear} free divisor; that is, $\textnormal{Der}(-\textnormal{log }D)$ is a free module and generated by vector fields 
which only have linear function coefficients with respect to the standard basis, {   provided: The quiver is a tree with a real Schur root 
and $\det \Delta$ which defines the discriminant $D$ 
is reduced 
(see Section \ref{sect-quiv} for details).} The first examples in this direction are \textit{Dynkin} quivers; that is, a quiver whose underlying unoriented graph is a Dynkin diagram of type ADE. A Dynkin quiver with a real Schur root as a dimension vector determines a linear free divisor (\cite{buch-mond}). It is well known that the Dynkin diagrams appear as the dual graph of the minimal resolution of rational surface singularities of complex surfaces of multiplicity $2$. Further examples of linear free divisors and a classification in dimension $N\leq 4$ was given in \cite{granger-mond-r-s}.

Recall that a singularity of a normal surface is \textit{rational} if the geometric genus of the surface is unchanged by a resolution of the singularity. The rational singularities of surfaces are classified by their multiplicity $m$ which equals $-Z^2$ where $Z$ is the Artin cycle of the resolution. When $m=2$, the dual graph of the minimal resolution of the rational singularity is one of the Dynkin diagrams, which are the underlying graphs of the Dynkin quivers.  So we naturally ask ``What are the other dual resolution graphs of singularities which may give linear free divisors?". Here we try to answer this question. 

By {   \cite[Theorem 3.9]{granger-mond-schulze},} linear free divisors can only be obtained from a quiver having the form of a tree. In this work, we are interested in rational singularities of surfaces as their dual resolution graphs are trees. We will refer to their resolution graphs as \textit{rational trees}. In particular, we consider rational singularities for which, in the Artin cycle, the coefficients of the exceptional curves corresponding to the vertices with weight $\geq 3$ are all equal to one. The singularities having such resolution graphs are called rational singularities with almost reduced Artin cycle (\cite{gustavsen}). The simplest singularities of surfaces of that type, which can also be seen as a generalisation of rational surface singularities of multiplicity $2$, are the rational singularities of multiplicity $3$, called rational triple points (or RTP for short). They are also determinantal singularities (\cite{wahl-equations}). Moreover, their dual resolution graphs contain Dynkin diagrams as subtrees. These are the simplest rational trees (quivers) answering our question above. 

We start by constructing a new root system for rational triple points using the results of \cite{M} and determine the number of roots for each RTP-singularity (cf. \cite{bourbaki}). In Section \ref{sect-free}, we show that, for each root, we obtain a linear free divisor (Theorem \ref{prop-lfd}). We also prove that this construction is independent of the orientation on the tree. Furthermore, using {    
the result in \cite{granger-mond-r-s}}, 
we deduce that linear free divisors defined by rational triple quivers satisfy the global logarithmic comparison theorem. In Section \ref{sect-quasi}, we {    give ideas on a generalisation } 
to rational singularities with almost reduced Artin cycle. 

In the 1970's Gabriel proved that a quiver is of finite representation type if and only if its underlying graph is a Dynkin diagram of type A, D or E (\cite{gabriel-2}). Then the classical McKay correspondence  described the bijection between the set of isomorphism classes of nontrivial irreducible representations of finite subgroups of $\textnormal{SL}(2,\mathbb C)$ and the vertices of the corresponding Dynkin diagrams of the quotient singularities (see~\cite{M}). In the case of rational singularities with almost reduced Artin cycle, even in the special case of RTP's, we do not have a Lie algebra corresponding to the singularity -- only some of the RTP's appear as quotient singularities (\cite{riemens}). 
In other words, the relation between the theory of representations of finite subgroups of our quivers and the minimal resolutions of these singularities is yet to be discovered. By constructing the root system for RTP singularities and  $\textnormal{Der}(- \textnormal{log }D)$, we hope to find Lie algebras for each RTP-singularity. 
 This is feasible since an explicit definition of reflexive modules for rational singularities is already given in \cite{wunram}. 


\section{Rational Singularities of Complex Surfaces}

Let $X$ be a germ of a normal surface embedded in $\mathbb C^N$ with a singularity at $0$. Let $\pi\colon \tilde X\rightarrow X$ be {   a} minimal resolution of $X$. The singularity of $X$ at $0$ is \textit{rational} if ${   \dim R^{1}\pi_{\ast}{\mathcal O}_{\tilde X}=0}$. This implies that the exceptional fibre $\pi^{-1}(0)$, denoted by $E:=\cup_{i=1}^{n}E_i$,  is a normal crossing divisor, each $E_i$ is a rational curve and the dual graph $\Gamma $ associated with $E$ is a weighted tree  with \textit{weight} $w_i:=-E_i^2$ at each vertex.  Recall that a {   non-zero} divisor $Y=\sum_{i}m_iE_i$ supported on $E$  (or equivalently, on $\Gamma $) is called a positive {    (resp. negative) divisor if $m_i\in {\mathbb N}$ (resp. $-m_i\in {\mathbb N}$) for all $i$.} A positive divisor $Z$ is called the Artin cycle of $\pi $ if 
{    $Z$ is the smallest positive divisor satisfying $(Z\cdot E_i)\leq 0$ for each $i$ (see \cite{artin}). } 
We will denote it by $Z=\sum_{i=1}^{n}a_iE_i$. 
{    For any rational surface singularity $(X, 0)$, the following $(1)$ and $(2)$ hold: } 
\newline $(1)$  For all  positive divisors $Y{    =\sum_{i}m_iE_i}$ that are supported on $E$, 
we have:
$$p_a(Y):=\frac{1}{2}\left[Y\cdot Y+\sum_{i=1}^{n}m_i(w_i-2) \right ]+1\leq 0.$$
\noindent $(2)$ The Artin cycle $Z$ of $\Gamma $ satisfies $Z\cdot Z=-m$ where $m$ is the multiplicity of $X$ at the singularity $0$.

\begin{remark} The Artin cycle $Z(\Gamma )$ of a rational tree $\Gamma $ with vertices $E_0,\ldots ,E_n$ is constructed by  Laufer's algorithm as follows (\cite{laufer}).  We put $Z_1:=\sum_{i=1}^{n}E_i$. If $(Z_1\cdot E_i)\leq 0$ for all $i=0,\ldots ,n$ then $Z_1$ is the Artin cycle $Z(\Gamma )$. If there exists an $E_{i_1}$ among the vertices of $\Gamma $ such that $(Z_1\cdot E_{i_1})>0$; in this case, we put $Z_2:=Z_1+E_{i_1}$ and  check  whether $(Z_2\cdot E_i)\leq 0$ for all $i$. At the $j$th step, we find a cycle $Z_j$, $(j\geq 1)$, which satisfies, either $(Z_j\cdot E_i)\leq 0$ for all $i$, in which case we put $Z(\Gamma):=Z_j$, or there is an irreducible component $E_{i_j}$ such that $(Z_j\cdot E_{i_j})>0$, then we put $Z_{j+1}=Z_j+E_{i_j}$. This process ends after a finite number of steps. The Artin cycle of $\Gamma $ is the first cycle $Z_k$ of this sequence such that $(Z_k\cdot E_i)\leq 0$ for all $i$.
\end{remark}

 Conversely, for a given tree $\Gamma $ with vertices $E_1,\ldots ,E_n$, let us assign a weight $w_i\in \mathbb N^*$ for each $E_i$.  Consider the incidence matrix $M(\Gamma )=(e_{ij})_{i,j=1,\ldots ,n}$, associated with $\Gamma $ such that $e_{ii}$ equals $-w_i$ and $e_{ij}$ is the number of edges between the vertices $E_i$ and $E_j$, which is $0$ or $1$ since $\Gamma$ is a tree. If $M(\Gamma )$ is negative definite then, by Grauert's theorem in \cite{grauert},  $\Gamma $ is the dual graph of the exceptional fibre of a resolution of a germ of a normal surface singularity where $E_i^2=-w_i$ for $i=1,\ldots, n$. When $w_i\geq 2$ for all $i$,  $\Gamma $ is the dual graph of the minimal resolution. 

\begin{defn}[\cite{artin}] Following this construction, if, in addition, {    a tree} 
$\Gamma$ satisfies 
{    $(1)$ and $(2)$ given above,} then the singularity is rational. In this case, $\Gamma$ is called a rational $m$-tuple tree. 
\end{defn}

  For example, Dynkin diagrams are the rational 2-tuple trees. These are the only rational trees having weight $2$ at each vertex.

\begin{prop}[\cite{artin, Le-tosun}] \label{prop-subtree} Any subtree of a rational $m$-tuple tree is a rational $m'$-tuple tree with $m'\leq m$.
\end{prop}

  Let $\Gamma $ be a rational $m$-tuple tree. Let us consider its vertices $E_i$ with weight $\geq 3$ and reindex them as $E_{i_0},E_{i_1},\ldots ,E_{i_k}$. Then we have 
\[ \Gamma -\lbrace E_{i_0},E_{i_1},\ldots ,E_{i_k}\rbrace =\coprod_j \Gamma_j \]
where $\Gamma_j$ is a Dynkin diagram for each $j$. {    A rational $m$-tuple tree} containing a unique vertex  with weight $w\geq 3$  (denoted by the symbol $\blacksquare$)  can be seen as in Figure \ref{fig-star}.

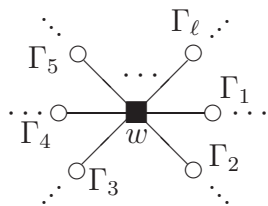
\begin{figure}[h!]
\setlength{\unitlength}{1mm}
\begin{picture}(120,30)(0,10)
\put(68.2,22){\line(-1,0){8}}
\put(68,18){$w$}
\put(68.25,21.25){\line(-1,-1){6}}
\put(67.8,20.6){$\blacksquare$}
\put(68.5,23){\line(-1,1){6}}
\put(70.25,21.25){\line(1,-1){6}}
\put(67.5,26){$\cdots$}
\put(70.25,23){\line(1,1){6}}
\put(70.45,22){\line(1,0){8}}
\put(77,30){\circle{2}}
\put(78.9,32){$.$}
\put(79.9,33){$.$}
\put(80.9,34){$.$}
\put(77,14.5){\circle{2}}
\put(78.9,11.9){$.$}
\put(79.9,10.9){$.$}
\put(80.9,9.9){$.$}
\put(61.5,14.3){\circle{2}}
\put(58.6,11.9){$.$}
\put(57.6,10.9){$.$}
\put(56.6,9.9){$.$}
\put(61.7,29.9){\circle{2}}
\put(58.6,32.2){$.$}
\put(57.6,33.2){$.$}
\put(56.6,34.2){$.$}
\put(79.6,22){\circle{2}}
\put(82,21){$\cdots$}
\put(59.1,22){\circle{2}}
\put(52.2,21){$\cdots$}
\put(80.5,24){$\Gamma_1$}
\put(79,15){$\Gamma_2$}
\put(63,11.75){$\Gamma_3$}
\put(54,18){$\Gamma_4$}
\put(55,28){$\Gamma_{5}$}
\put(74,33){$\Gamma_\ell$}
\end{picture}
\caption{Star rational tree.}
\label{fig-star}
\end{figure}

 In addition, we have the following property of rational trees which bounds the number of the subtrees $\Gamma_j$.

\begin{prop}[\cite{spivakovsky}] \label{prop-valency} If $\Gamma$ is a rational tree, then each vertex $E_i$ satisfies \[v_i \leq w_i+1\]
where  $v_i$ is the number of vertices adjacent to $E_i$ in $\Gamma$, called the \textit{valency} of $E_i$ and, $w_i$ is the weight of $E_i$ given by $-E_i^2$.
\end{prop}

 Furthermore, any given tree can be weighted in such a way that it becomes rational. See \cite{Le-tosun} for the details and the glueing conditions of rational trees.

\section{Triple Root System}\label{sect-tripleroots}

The classification of rational 3-tuple trees is given by Artin in \cite{artin}.  They are of the form given in Figure \ref{fig-star} such that $w=3$ and $\ell \leq 3$ where $\ell$ is the valency of the \textit{central vertex}, i.e. the vertex denoted by {    $\blacksquare$.} The corresponding rational singularities are singularities of surfaces embedded in $\mathbb{C}^4$ and defined by three equations as listed in \cite{tjurina}.  

  In the sequel, a rational $3$-tuple tree $\Gamma $ will be called an RTP-tree.  The Artin cycle $Z$ of an RTP-tree  satisfies $Z^2=-3$ and $p_a(Z)=0$. Hence, by the formulas given in $(1)$ and $(2)$ above, an RTP-tree  contains a unique vertex with weight $3$ and attached to some Dynkin diagrams. Moreover, {    by $Z^2=-3$ and $p_a(Z)=0$, } 
the coefficient of the vertex with weight $3$ in the Artin cycle $Z$ is equal to $1$. 

\begin{defn}\label{def-tripleroot}\rm Let $\Gamma $ be an RTP-tree and $\mathcal{G}$ be the set of all positive and negative divisors supported on $E$. Consider the set {    
\[R(\Gamma ):=\left\{ Y\in\mathcal{G} \mid (Y\cdot Y)=-2\ \ \textnormal{or} \ \ (Y\cdot Y)=-3\right\}. \]
The set $R(\Gamma )$ is called  the {\it triple root system}. }
\end{defn}

  We call an element of $R(\Gamma)$ a {\it root}. We also call each $E_i$ a {\it simple (positive) root} of $R(\Gamma)$.
Note that for Dynkin diagrams, the set of divisors $Y$ in $\mathcal{G}$ with the property  $(Y\cdot Y)=-2$ form the root system (\cite{pinkham}, \cite{M}, compare with \cite{bourbaki}).

 From now on, we will denote by $E_0$  the unique vertex in an RTP-tree with weight $3$  and the coefficient of $E_0$ in a root $Y$ by $a_0$.

\begin{lemma}\label{lemma:-2}
If $Y =\sum a_iE_i$ is a positive divisor on $\Gamma$ then $(Y\cdot Y)\leq -2$.
\end{lemma}

\begin{proof} It easily follows from the fact that $p(Y)\leq 0$ (see also \cite{artin}).\end{proof}

\begin{lemma}\label{lem-suppy}
If $Y \in R(\Gamma)$, then $Y$ is either a positive or negative divisor. 
\end{lemma}

\begin{proof} Suppose that $Y = Y_{1}-Y_{2}$ such that
$Y_{1}$ and $Y_{2}$ are positive and without common components.
Then $(Y\cdot Y)=(Y_{1}\cdot Y_{1})-2(Y_{1}\cdot Y_{2})+(Y_{2}\cdot Y_{2})$.
Note that $(Y_{1}\cdot Y_{1}) \le 0$, $(Y_{2}\cdot Y_{2}) \le 0$, and
$(Y_{1}\cdot Y_{2}) \ge 0$. If $Y_1 \neq 0$ and $Y_2 \neq 0$, then $(Y_1\cdot Y_1) \le -2$ and
$(Y_2\cdot Y_2) \le -2$ by Lemma \ref{lemma:-2}. However this contradicts
$(Y\cdot Y) = -2$ or $-3$. Hence $Y_1=0$ or $Y_2=0$. \end{proof}

\begin{lemma}\label{support}
If $Y = \sum a_i E_i \in R(\Gamma)$, then ${\rm Supp}(Y)$ is connected where $\textnormal{Supp}(Y)$ is the support of  $Y$ which is the set of $E_i$'s for which $a_i\neq 0$.
\end{lemma}

\begin{proof} It suffices to show that the subgraph of $\Gamma$ given by ${\rm Supp}(Y)$  is a
connected tree for a positive divisor $Y$.
Suppose that ${\rm Supp}(Y)$ is not connected.
Then we can write $Y=Y_{1}+Y_{2}$ such that $Y_{1}$ and $Y_{2}$ are
non-zero positive cycles and
$(Y_{1}\cdot Y_{2}) =0$. However, $(Y\cdot Y)=(Y_{1}\cdot Y_{1})+(Y_{2}\cdot Y_{2}) \leq -4$, which
is a contradiction. 
\end{proof}

\begin{prop}\label{prop-a0}
Let $Y\in R(\Gamma )$. {    We have $(Y\cdot Y)=-2$ if and only if $a_0=0$. 
We also have $(Y\cdot Y)=-3$ if and only if $a_0=\pm 1$. }  
\end{prop}

{    
\begin{proof} 
It suffices to prove the statement for a positive root $Y \in R(\Gamma)$.  
Note that $p(Y) = 1+ ((Y\cdot Y)+a_0)/2 \le 0$. 
We have $(Y\cdot Y) \le -2-a_0$. 
Since $(Y\cdot Y)= -2$ or $-3$, $a_0 = 0$ or $1$. 
If $(Y\cdot Y)=-2$, then $a_0=0$. 
Conversely, if $a_0 = 0$, then 
$Y$ is supported on a subtree of $\Gamma $ which is a Dynkin diagram 
by Lemma \ref{support}. 
This implies $(Y\cdot Y)=-2$. 
Hence we see that $(Y\cdot Y)=-2$ if and only if $a_0=0$. 
By taking the contraposition, we also see that 
$(Y\cdot Y)=-3$ if and only if $a_0=1$. 
\end{proof}
}

%
%

\begin{defn}
Let $\Gamma$ be an RTP-tree. We will call central vertex its vertex with weight $3$. 
Let $\Gamma_0$ denote the tree obtained from $\Gamma$ by changing the weight of the central vertex of $\Gamma$ into the weight $2$.
We say that $\Gamma_0$ is the {\it underlying diagram} of $\Gamma$. The underlying diagram $\Gamma_0$ is {\it classical} if it is of type $A, D$ or $E$.
If $\Gamma_0$ of a given RTP-tree  $\Gamma$ is classical then we call $\Gamma$ a triple Dynkin diagram.
\end{defn}

 We denote the simple roots of $R(\Gamma)$ and $R(\Gamma_0)$ by the same notation $E_i$, while $E_0$ denotes the simple root corresponding to the central vertex. 
Accordingly, $a_0$ denotes the coefficient of the central vertex in an element $Y$ in $R(\Gamma)$ (and in  $R(\Gamma_0)$).
By this notation, $R(\Gamma)$ and $R(\Gamma_0)$ can be embedded in $\oplus {\mathbb Z}E_i$.

\begin{prop}
Let $\Gamma$ be a triple Dynkin diagram. Then $R(\Gamma) \subseteq R(\Gamma_0)$.
\end{prop}

\begin{proof}
For each $Y \in \sum a_i E_i \in \oplus {\mathbb Z}E_i$,
we have $(Y\cdot Y)=(Y\cdot Y)_{\Gamma_0}-a_0^2$, where
$(Y\cdot Y)_{\Gamma_0}$ denotes the self-intersection of $Y$
with respect to $R(\Gamma_0)$.
If $(Y\cdot Y)=-2$, then $(Y\cdot Y)_{\Gamma_0} \ge -2$, and hence
$(Y\cdot Y)_{\Gamma_0} = -2$ and $a_0=0$.
If $(Y\cdot Y)=-3$, then $(Y\cdot Y)_{\Gamma_0} \ge -3$.
Because $(Y\cdot Y)_{\Gamma_0}$ is even, $(Y\cdot Y)_{\Gamma_0} = -2$
and $a_0 = \pm 1$.
Therefore $R(\Gamma) \subseteq R(\Gamma_0)$. \end{proof}


\begin{corollary}
Let $\Gamma$ be a triple Dynkin diagram. Let $Z(\Gamma_0)=\sum a_i E_i$ be the highest root  of $R(\Gamma_0)$.
If the coefficient $a_0 = 1$, then we have $R(\Gamma) = R(\Gamma_0)$. 
\end{corollary}

\begin{proof}
We only need to show that
each positive root $Y$ of $R(\Gamma_0)$ is contained in $R(\Gamma)$.
Since $Y \le Z(\Gamma_0)$, the coefficient $a_0$ of $Y$ is
less than $1$. Then
$(Y\cdot Y)=(Y\cdot Y)_{\Gamma_0}-a_0^2 \ge -2-1 = -3$.
On the other hand, $(Y\cdot Y) \le -2$.
Thus we have proved that $Y \in R(\Gamma)$.
\end{proof}

%
%

 We introduce examples of triple Dynkin diagrams as follows. Consider a tree $\Gamma$ of type $A, D$ or $E$ and denote by $\Gamma_i$ the triple
Dynkin diagram obtained by replacing $E_i$ by a vertex of weight $3$. For example, $A_{n,2}$ is the diagram obtained from $A_n$  (see Figure \ref{fig-An}) by replacing $E_2$ by a vertex of weight $3$.

\begin{example}\label{example:An} Consider the Dynkin diagram of type $A_n$ in Figure \ref{fig-An}. It is easy to see that $R(A_n)=R(A_{n,i})$ for each $i$.
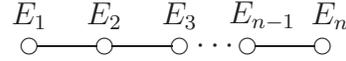
\begin{figure}[h!]
\setlength{\unitlength}{1mm}
\begin{picture}(125,10)(-10,30)
\put(40,34){\makebox(0,0){$E_1$}}
\put(40,30){\circle{2}}
\put(50,30){\circle{2}}
\put(50,34){\makebox(0,0){$E_2$}}
\put(41,30){\line(1,0){8}}
\put(51,30){\line(1,0){8}}
\put(60,30){\circle{2}}
\put(60,34){\makebox(0,0){$E_3$}}
\put(65,30){\makebox(0,0){$\cdots$}}
\put(69,30){\circle{2}}
\put(71,34){\makebox(0,0){$ E_{n-1}$}}
\put(70,30){\line(1,0){8}}
\put(79,30){\circle{2}}
\put(80,34){\makebox(0,0){$E_n$}}
\end{picture}
\caption{Dynkin diagram of type $A_n$.}
\label{fig-An}
\end{figure}
\end{example}

\begin{example}\label{example:Dn} Triple Dynkin diagrams of type $D_n$, as shown in Figure \ref{fig-Dn}, yield the following relations between the root systems:

\[R(D_{n,n-2})\subseteq R(D_{n, n-3})\subseteq \cdots \subseteq R(D_{n,1}) = R(D_{n,n})=R(D_{n,{n-1}})=R(D_n).\]

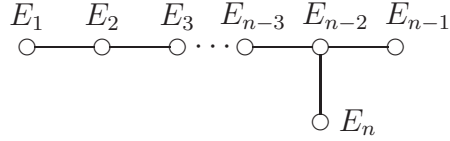
\begin{figure}[!h]
\setlength{\unitlength}{1mm}
\begin{picture}(125,20)(0,18)
\put(40,34){\makebox(0,0){$E_1$}}
\put(40,30){\circle{2}}
\put(50,30){\circle{2}}
\put(50,34){\makebox(0,0){$E_2$}}
\put(41,30){\line(1,0){8}}
\put(51,30){\line(1,0){8}}
\put(60,30){\circle{2}}
\put(60,34){\makebox(0,0){$ E_3$}}
\put(65,30){\makebox(0,0){$\cdots$}}
\put(69,30){\circle{2}}
\put(70,34){\makebox(0,0){$E_{n-3}$}}
\put(70,30){\line(1,0){8}}
\put(79,30){\circle{2}}
\put(81,34){\makebox(0,0){$E_{n-2}$}}
\put(80,30){\line(1,0){8}}
\put(89,30){\circle{2}}
\put(92,34){\makebox(0,0){$ E_{n-1}$}}
\put(79,29){\line(0,-1){8}}
\put(79,20){\circle{2}}
\put(84,20){\makebox(0,0){$E_{n}$}}
\end{picture}
\caption{Dynkin diagram of type $D_n$.}
\label{fig-Dn}
\end{figure}
\end{example}

 The last equalities on the right hand side are clear as their coefficients in the Artin cycle (the highest root) of $D_n$ is $1$, so their weights don't change the roots obtained for the Dynkin diagram of type $D_n$. If we put $w_{i_0}=3$ for $2\leq i_0\leq n-2${   , then} the coefficient $a_{i_0}$ becomes $1$ in the highest root. In this case, the {    Artin} cycle $Z(D_{n,i_0-1})$  is smaller than the {   Artin} cycle $Z(D_{n,i_0})$. Then we obtain a smaller set of roots $R(D_{n,i_0})$ which is included in the set of roots of $R(D_{n,i_0-1})$.
\begin{example}\label{example:E6}
For the Dynkin diagram $E_6$ pictured in Figure \ref{fig-E6}, after a smimilar discusion on the place of the vertex with weight $3$, we obtain:
{    
\[\xymatrixrowsep{10pt}\xymatrix{  &R(E_{6,2}) \ar[r] & R(E_{6,6}) \ar[r] & R(E_{6,1})=R(E_{6,5})=R(E_6)\\
  R(E_{6,3}) \ar[ur] \ar[r] & R(E_{6,4}) \ar[ur] & &}.\] } 

\begin{figure}[h!]
\setlength{\unitlength}{1mm}
\begin{picture}(125,15)(-10,18)
\put(40,34){\makebox(0,0){$E_1$}}
\put(40,30){\circle{2}}
\put(50,30){\circle{2}}
\put(50,34){\makebox(0,0){$E_2$}}
\put(41,30){\line(1,0){8}}
\put(51,30){\line(1,0){8}}
\put(60,30){\circle{2}}
\put(60,34){\makebox(0,0){$E_3$}}
\put(61,30){\line(1,0){8}}
\put(70,30){\circle{2}}
\put(70,34){\makebox(0,0){$E_4$}}
\put(71,30){\line(1,0){8}}
\put(80,30){\circle{2}}
\put(80,34){\makebox(0,0){$E_5$}}
\put(60,29){\line(0,-1){8}}
\put(60,20){\circle{2}}
\put(64,20){\makebox(0,0){$E_6$}}
\end{picture}
\caption{Dynkin diagram of type $E_6$.}
\label{fig-E6}
\end{figure}
\end{example}

\begin{example}\label{example:E7} Consider the Dynkin diagram $E_7$ in Figure \ref{fig-E7}.

\begin{figure}[h!]
\setlength{\unitlength}{1mm}
\begin{picture}(125,25)(-5,18)
\put(40,34){\makebox(0,0){$E_1$}}
\put(40,30){\circle{2}}
\put(50,30){\circle{2}}
\put(50,34){\makebox(0,0){$E_2$}}
\put(41,30){\line(1,0){8}}
\put(51,30){\line(1,0){8}}
\put(60,30){\circle{2}}
\put(60,34){\makebox(0,0){$ E_3$}}
\put(61,30){\line(1,0){8}}
\put(70,30){\circle{2}}
\put(70,34){\makebox(0,0){$E_4$}}
\put(71,30){\line(1,0){8}}
\put(80,30){\circle{2}}
\put(80,34){\makebox(0,0){$ E_5$}}
\put(81,30){\line(1,0){8}}
\put(90,30){\circle{2}}
\put(90,34){\makebox(0,0){$E_6$}}
\put(60,29){\line(0,-1){8}}
\put(60,20){\circle{2}}
\put(64,20){\makebox(0,0){$E_7$}}
\end{picture}
\caption{Dynkin diagram of type $E_7$.}
\label{fig-E7}
\end{figure}

 We have the following inclusion of the root systems. {    
\[\xymatrixrowsep{10pt}\xymatrix{ &R(E_{7,2}) \ar[r] &R(E_{7,7}) \ar[r] & R(E_{7,1}) \ar[r] &R(E_{7,6})=R(E_7)\\
R(E_{7,3}) \ar[ur] \ar[r]  & R(E_{7,4}) \ar[ur] \ar[r] & R(E_{7,5}) \ar[ur] & &}.\] } 
\end{example}

\begin{example}\label{example:E8} Consider the Dynkin diagram $E_8$ shown in Figure \ref{fig-E8}. 
\begin{figure}[h!]
\setlength{\unitlength}{1mm}
\begin{picture}(130,20)(0,18)
\put(40,34){\makebox(0,0){$E_1$}}
\put(40,30){\circle{2}}
\put(50,30){\circle{2}}
\put(50,34){\makebox(0,0){$ E_2$}}
\put(41,30){\line(1,0){8}}
\put(51,30){\line(1,0){8}}
\put(60,30){\circle{2}}
\put(60,34){\makebox(0,0){$ E_3$}}
\put(61,30){\line(1,0){8}}
\put(70,30){\circle{2}}
\put(70,34){\makebox(0,0){$E_4$}}
\put(71,30){\line(1,0){8}}
\put(80,30){\circle{2}}
\put(80,34){\makebox(0,0){$E_5$}}
\put(81,30){\line(1,0){8}}
\put(90,30){\circle{2}}
\put(90,34){\makebox(0,0){$E_6$}}
\put(91,30){\line(1,0){8}}
\put(100,30){\circle{2}}
\put(100,34){\makebox(0,0){$ E_7$}}
\put(60,29){\line(0,-1){8}}
\put(60,20){\circle{2}}
\put(64,20){\makebox(0,0){$E_8$}}
\end{picture}
\caption{Dynkin diagram of type $E_8$.}
\label{fig-E8}
\end{figure}

 Then we have
\[\xymatrixrowsep{10pt}\xymatrix{ & & & R(E_{8,6}) \ar[dr] & & \\
 R(E_{8,3}) \ar[r] \ar[dr] & R(E_{8,4}) \ar[dr] \ar[r] & R(E_{8,5}) \ar[r] \ar[ur] & R(E_{8,1}) \ar[r] & R(E_{8,7}) \ar[r] & R(E_8)\\
 & R(E_{8,2}) \ar[r] & R(E_{8,8}) \ar[ur] & & &}.\]
 \end{example}

  Now we will compute the number of roots in each triple root system corresponding to the rational triple trees classified by Artin (\cite{artin}) (see also Figures \ref{fig-Anmk}-\ref{fig-E82} where $\blacksquare$ is the vertex with weight $3$ in each tree). Here we refer to them using the notation $A_{n,m,k}$, $B_{m,n}$, $C_{m,n}$, $D_{n,5}$, $F_n$, $H_n$ following Tjurina's work (\cite{tjurina}) and $E_{7,1}$, $E_{8,1}$, $E_{8,2}$ for the rest which were unlabelled in either two articles. We will denote the number of elements in the set $R(\Gamma)$ by $|R(\Gamma)|$.

\begin{prop}\label{prop:anmk}
The number of roots in $R(A_{n,m,k})$ is equal to
\[n^2+m^2+k^2+n+m+k+2(n+1)(m+1)(k+1)\]
where the corresponding graph $A_{n,m,k}$ is a tree with $n+m+k+1$ vertices as shown in Figure \ref{fig-Anmk}.

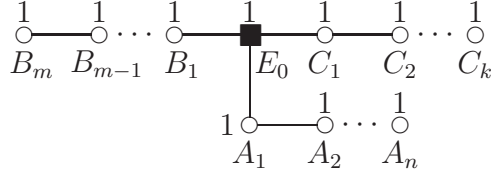
\begin{figure}[h!]
\setlength{\unitlength}{1mm}
\begin{picture}(120,25)(-10,5)
\put(31,18){\makebox(0,0){$ B_m$}}
\put(30,22){\circle{2}}
\put(30,25){\makebox(0,0){$1$}}
\put(40,22){\circle{2}}
\put(40,25){\makebox(0,0){$ 1$}}
\put(41,18){\makebox(0,0){$ B_{m-1}$}}
\put(31,22){\line(1,0){8}}
\put(45,22){\makebox(0,0){$\cdots$}}
\put(50,22){\circle{2}}
\put(50,25){\makebox(0,0){$ 1$}}
\put(51,18){\makebox(0,0){$ B_1$}}
\put(51,22){\line(1,0){8}}
\put(58.5,20.5){$\blacksquare$}
\put(60,25){\makebox(0,0){$ 1$}}
\put(63,18){\makebox(0,0){$ E_0$}}
\put(61,22){\line(1,0){8}}
\put(70,22){\circle{2}}
\put(70,18){\makebox(0,0){$ C_1$}}
\put(70,25){\makebox(0,0){$ 1$}}
\put(71,22){\line(1,0){8}}
\put(80,22){\circle{2}}
\put(80,18){\makebox(0,0){$ C_2$}}
\put(80,25){\makebox(0,0){$ 1$}}
\put(85,22){\makebox(0,0){$\cdots$}}
\put(90,22){\circle{2}}
\put(90,18){\makebox(0,0){$ C_k$}}
\put(90,25){\makebox(0,0){$ 1$}}
\put(60,10){\circle{2}}
\put(57,10){\makebox(0,0){$ 1$}}
\put(60,21){\line(0,-1){10}}
\put(60,6){\makebox(0,0){$ A_1$}}
\put(61,10){\line(1,0){8}}
\put(70,10){\circle{2}}
\put(70,6){\makebox(0,0){$ A_2$}}
\put(70,13){\makebox(0,0){$ 1$}}
\put(75,10){\makebox(0,0){$\cdots$}}
\put(80,10){\circle{2}}
\put(80,6){\makebox(0,0){$ A_n$}}
\put(80,13){\makebox(0,0){$ 1$}}
\end{picture}
\caption{Triple diagram of type $A_{n,m,k}$.}
\label{fig-Anmk}
\end{figure}
\end{prop}

 We need the following lemma to prove Proposition \ref{prop:anmk}.

\begin{lemma}\label{lemma:inequality}
Let us consider the Dynkin diagram $A_n$ (Figure \ref{fig-An}). For any positive cycle $Y=\sum_{i=1}^{n} a_i E_i$, we have
$(Y\cdot Y) \le -a_n^2-1$.
In particular, $(Y\cdot Y) \le -2a_n$.
\end{lemma}

\begin{proof} We prove the statement by induction on $n$. If $n=1$, then $(Y\cdot Y)=-2a_1^2 \le -a_1^2-1$.

 Assume that $n \ge 2$. Suppose that the statement is true in the case $n-1$. Let $Y'=\sum_{i=1}^{n-1} a_i E_i$ so that $Y=Y'+a_n E_n$. By the hypothesis, $(Y'\cdot Y') \le -a_{n-1}^2-1$. Hence 
\begin{eqnarray*}
(Y\cdot Y)&=&(Y'\cdot Y')+2(Y'\cdot a_{n}E_{n})-2a_{n}^2 \\
&\leq &
-a_{n-1}^2-1+2a_{n-1}a_{n}-2a_{n}^2 \\ 
&= &-(a_{n-1}-a_{n})^2-a_{n}^2-1 \\
& \le &  -a_{n}^2-1. \end{eqnarray*}

 On the other hand, the second claim follows from the fact that $-x^2-1 \le -2x$ for any $x$. This completes the proof.
\end{proof}

\begin{proof}[Proof of Proposition \ref{prop:anmk}] Let $Y$ be a positive root in $R(A_{n, m, k})$. We can write
$Y=A+B+C+a_0 E_0$, where $A=\alpha_{1} A_{1} + \cdots +\alpha_{n}A_{n}$, $B= \beta_{1} B_{1} + \cdots + \beta_{m}B_{m}$,
and $C=\gamma_{1} C_{1} + \cdots + \gamma_{k}C_{k}$. Then $a_0=0$ or $1$. If $a_0=0$, then $Y=A$, $B$, or $C$ since ${\rm Supp}(Y)$ is connected. In this case the number of positive roots is equal to
\begin{equation}\label{eq-anmk}\frac{1}{2}\left (|R(A_{n})|+|R(A_m)|+ |R(A_k)|\right) = \frac{1}{2}\left(n^2+n+m^2+m+k^2+k\right).\end{equation}

 Now let us assume that $a_0=1$. Note that 
\[(Y\cdot Y)=(A\cdot A)+2\alpha_1+ (B\cdot B)+2\beta_1+(C\cdot C)+2\gamma_1 -3\]
By Lemma \ref{lemma:inequality}, $(A\cdot A) +2\alpha_1 \le 0$, $(B\cdot B)+ 2\beta_1 \le 0$, and $(C\cdot C)+2\gamma_1 \le 0$. 
Since $Y$ is a root of $A_{n, m, k}$, we must have $(A\cdot A) +2\alpha_1 = 0$, $(B\cdot B) + 2\beta_1 =
0$ and $(C\cdot C)+2\gamma_1 = 0$. 
{    If $A$ is a positive divisor, then} $-2\alpha_1 = (A\cdot A) \le -\alpha_1^2-1$,
and hence $\alpha_1=1$ {    and $(A\cdot A)=-2$. 
Thereby $A \in R(A_{n}) \cup \{ 0 \}$.} Similarly, {    we have 
$B \in R(A_{m}) \cup \{ 0 \}$ and $C \in R(A_{k}) \cup \{ 0 \}$. }
%
{   We} can easily verify that the number of positive roots with $a_0=1$ is equal to $(n+1)(m+1)(k+1)$ since ${\rm Supp} (Y)$ is connected.

 The sum in (\ref{eq-anmk}) and  $(n+1)(m+1)(k+1)$ add up to the number of positive roots. Clearly, we get the same number of negative roots. This concludes the proof. \end{proof}

\begin{prop}\label{prop:B_mn}
The number of roots in $R(B_{m,n})$ is 
\begin{equation}\label{eq-RBmn} n(n+1)(m+1)+m(m+1)+n(n+1)
\end{equation} with $m,n\geq 0$.

\begin{figure}[h!]
\setlength{\unitlength}{1mm}
\begin{picture}(125,20)(-5,12)
\put(31,18){\makebox(0,0){$E_m$}}
\put(30,22){\circle{2}}
\put(40,22){\circle{2}}
\put(41,18){\makebox(0,0){$E_{m-1}$}}
\put(31,22){\line(1,0){8}}
\put(45,22){\makebox(0,0){$\cdots$}}
\put(50,22){\circle{2}}
\put(51,18){\makebox(0,0){$E_1$}}
\put(51,22){\line(1,0){8}}
\put(58.8,20.5){$\blacksquare$}
\put(60,18){\makebox(0,0){$E_0$}}
\put(61,22){\line(1,0){8}}
\put(70,22){\circle{2}}
\put(71,18){\makebox(0,0){$F_2$}}
\put(71,22){\line(1,0){8}}
\put(80,22){\circle{2}}
\put(81,18){\makebox(0,0){$ F_3$}}
\put(85,22){\makebox(0,0){$\cdots$}}
\put(90,22){\circle{2}}
\put(91,18){\makebox(0,0){$F_{n-1}$}}
\put(91,22){\line(1,0){8}}
\put(100,22){\circle{2}}
\put(101,18){\makebox(0,0){$ F_n$}}
\put(70,32){\circle{2}}
\put(70,31){\line(0,-1){8}}
\put(74,32){\makebox(0,0){$F_1$}}
\end{picture}
\caption{Triple diagram of type $B_{m,n}$.}
\label{fig-Bmn}
\end{figure}
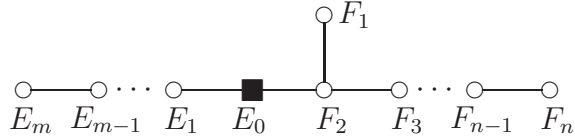
\end{prop}

\begin{proof} We consider the diagram in Figure \ref{fig-Bmn}. Let $Y\in R(B_{m,n})$ be given by $Y=\sum_{i=0}^{m} a_i E_i + \sum_{j=1}^n b_j F_j$. 
First we study the case $a_0=0$. Since ${\rm Supp}(Y)$ is connected,
\[Y \in R(A_m) := R(B_{m, n}) \cap \oplus_{i=1}^{m} {\mathbb Z}E_i\]
or  \[Y \in R(A_n) := R(B_{m, n}) \cap \oplus_{j=1}^{n} {\mathbb Z}F_j.\]
Then, the number of all roots with $a_0=0$ is  
$|R(A_m)|+|R(A_n)| = m(m+1)+n(n+1)$.

 Next assume that $a_0=1$. Put $D_1=\sum_{i=1}^{m}a_iE_i$ and
$D_2 = E_0 + \sum_{j=1}^{n}b_jF_j$. Note that
\[-3=(Y\cdot Y)=(D_1+D_2)^2=D_1^2+2a_1+D_2^2\] and that $D_1^2+2a_1 \le 0$ 
{    by Lemma \ref{lemma:inequality}. By Proposition \ref{prop-a0},  
$D_2^2 = -3$ and hence $D_1^2=-2a_1$. } 
Again by Lemma \ref{lemma:inequality}, we have $D_1^2=-2$ or $D_1=0$
because $D_1^2 \le -a_1^2-1 \le -2a_1 = D_1^2$.
Thus $D_1 \in  R(A_m)$
or $D_1=0$, and $D_2 \in
R(D_{n+1, {   n+1}}) = R(B_{m, n}) \cap {\mathbb Z} E_0 \oplus
(\oplus_{j=1}^{n} {\mathbb Z}F_j)$.
By the connectedness of ${\rm Supp}(Y)$, we have
\[ \left\{ Y \in R(B_{m, n}) \mid a_0=  1  \right\}=\left\{ {    D_1+D_2} \: 
\begin{array}{|c}
  D_1=0 \mbox{ or } D_1 \in R(A_m) \mbox{ and } a_1 > 0 \\
  D_2 = E_0 + \sum_{j=1}^{n}b_jF_j \in R(D_{n+1, {   n+1}})
\end{array} \right\}.  \]
Note that $|R(D_{n+1, {   n+1}})| = |R(D_{n+1})|$ by Example \ref{example:Dn}.
The number of elements in $\{ D_2 =  E_0 + \sum_{j=1}^{n}b_jF_j \in R(D_{n+1, {   n+1}})
\}$ is equal to the number of elements in $R(D_{n+1})_{+} \setminus \{ D \in R(D_{n+1})_{+} \mid
{\rm Supp} (D) \subset \{ F_j \mid j=1, 2, \ldots, n \} \}$, which is
$|R(D_{n+1})_{+}| - |R(A_n)_{+}|$. Here we denote by $R(\  )_{+}$
the positive roots in $R(\ )$.
It is easy to check that the number of elements in $\{ D_1 \mid D_1=0 \mbox{ or }
D_1 \in R(A_m) \mbox{ with } a_1 > 0\}$ is $m+1$.
Therefore 
\begin{eqnarray*}|\{ Y \in R(B_{m, n}) : a_0=  1  \}| &=&
(m+1)\left(|R(D_{n+1})_{+}|-|R(A_n)_{+}|\right)\\
&=& (m+1)n(n+1)/2.\end{eqnarray*}

 Similarly, the number of all roots with $a_0=-1$ is equal to
$n(n+1)(m+1)/2$. Summing up all of the numbers above, we obtain (\ref{eq-RBmn}).
\end{proof}

%
%
%
\begin{prop}
The number of roots in $R(C_{m,n})$ is $2m^2+4mn+n^2+2m+n$.

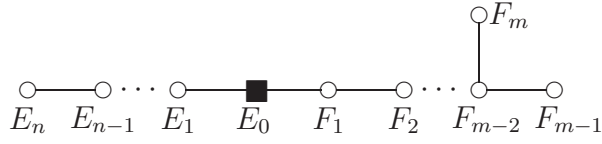
\begin{figure}[h!]
\setlength{\unitlength}{1mm}
\begin{picture}(120,20)(0,15)
\put(30,18){\makebox(0,0){$E_n$}}
\put(30,22){\circle{2}}
\put(40,22){\circle{2}}
\put(40,18){\makebox(0,0){$E_{n-1}$}}
\put(31,22){\line(1,0){8}}
\put(45,22){\makebox(0,0){$\cdots$}}
\put(50,22){\circle{2}}
\put(50,18){\makebox(0,0){$E_1$}}
\put(51,22){\line(1,0){8}}
\put(58.8,20.5){$\blacksquare$}
\put(60,18){\makebox(0,0){$E_0$}}
\put(61,22){\line(1,0){8}}
\put(70,22){\circle{2}}
\put(70,18){\makebox(0,0){$F_1$}}
\put(71,22){\line(1,0){8}}
\put(80,22){\circle{2}}
\put(85,22){\makebox(0,0){$\cdots$}}
\put(90,22){\circle{2}}
\put(80,18){\makebox(0,0){$F_2$}}
\put(91,18){\makebox(0,0){$F_{m-2}$}}
\put(91,22){\line(1,0){8}}
\put(100,22){\circle{2}}
\put(102,18){\makebox(0,0){$F_{m-1}$}}
\put(90,32){\circle{2}}
\put(90,31){\line(0,-1){8}}
\put(94,32){\makebox(0,0){$F_m$}}
\end{picture}
\caption{Triple diagram of type $C_{m,n}$.}
\label{fig-Cmn}
\end{figure}
\end{prop}

\begin{proof} The tree $C_{m,n}$ can be seen as the
glueing of rational double trees of type $A_{n}$ and
$D_{m}$ by the vertex $E_0$ (see Figure \ref{fig-Cmn}).
Let $Y=\sum_{i=0}^{n}a_iE_i +
\sum_{i=1}^{m}b_j{   F_j} \in R(C_{m, n})$.
Then $a_0=0$ or $\pm 1$.
If $a_0=0$ then $Y$ is supported on $A_{n}$ or $D_{m}$. If it is
supported on $A_{n}$, we have $|R(A_{n})|=n(n+1)$
and, if it is supported on $D_{m}$ we have
$|R(D_{m})|=2m(m-1)$.

 Let us consider the case $a_0=1$.
Note that the underlying diagram of $C_{m, n}$ is of type $D_{m+n+1}$.
Comparing the bilinear forms of $C_{m, n}$ and $D_{m+n+1}$, we have
$(Y\cdot Y)=(Y\cdot Y)_{D_{m+n+1}} -a_{0}^2$, where
we denote by $(\ \cdot \  )_{D_{m+n+1}}$ the bilinear form of
$D_{m+n+1}$.
Since $a_0=1$ and $(Y\cdot Y)=-3$, $(Y\cdot Y)_{D_{m+n+1}}=-2$.
Hence $Y$ can be regarded as a positive root of $R(D_{m+n+1})$
with $a_0=1$. Now let us calculate the number of such roots.
Let $\Delta =\lbrace \pm e_i\pm e_j \mid 1\leq i\not =j\leq m+n+1\rbrace $
be the set of roots in $D_{m+n+1}$. The positive simple roots are
$\alpha_i=e_i-e_{i+1}$, $1\leq i<m+n+1$, and $\alpha_{m+n+1}={    e_{m+n}+e_{m+n+1}}$.
The set of all positive roots satisfying the condition $a_0=1$ is
\[\{e_{i} \pm e_{j} \mid 1 \le i \le n+1 < j \le m+n+1  \}.\]
Hence the number of roots with $a_0=1$ is $2mn+2m$.
We also obtain the same result for $a_0=-1$.
Therefore,
\[|R(C_{m,n})|=n(n+1)+2m(m-1)+4mn+4m=2m^2+4mn+n^2+2m+n.\]
This concludes the proof. \end{proof}


%
%
%
\begin{prop}
The number of roots in $R(D_{n,5})$ is $n^2+33n+72$.

\begin{figure}[h!]
\setlength{\unitlength}{1mm}
\begin{picture}(125,20)(-5,15)
\put(30,18){\makebox(0,0){$E_n$}}
\put(30,22){\circle{2}}
\put(40,22){\circle{2}}
\put(40,18){\makebox(0,0){$E_{n-1}$}}
\put(31,22){\line(1,0){8}}
\put(45,22){\makebox(0,0){$\cdots$}}
\put(50,22){\circle{2}}
\put(50,18){\makebox(0,0){$E_1$}}
\put(51,22){\line(1,0){8}}
\put(58.8,20.5){$\blacksquare$}
\put(60,18){\makebox(0,0){$E_0$}}
\put(61,22){\line(1,0){8}}
\put(70,22){\circle{2}}
\put(71,18){\makebox(0,0){$F_1$}}
\put(71,22){\line(1,0){8}}
\put(80,22){\circle{2}}
\put(81,18){\makebox(0,0){$F_2$}}
\put(81,22){\line(1,0){8}}
\put(90,22){\circle{2}}
\put(91,18){\makebox(0,0){$F_3$}}
\put(91,22){\line(1,0){8}}
\put(100,22){\circle{2}}
\put(101,18){\makebox(0,0){$F_4$}}
\put(80,32){\circle{2}}
\put(80,31){\line(0,-1){8}}
\put(84,32){\makebox(0,0){$F_5$}}
\end{picture}
\caption{Triple diagram of type $D_{n,5}$.}
\label{fig-Dn5}
\end{figure}
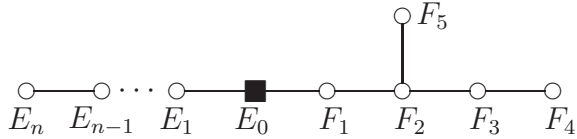
\end{prop}

\begin{proof} The tree $D_{n,5}$ can be seen as the glueing of the rational trees $A_n$ and $D_5$ by the vertex $E_0$ (see Figure \ref{fig-Dn5}) .
If $n=0$, by Example \ref{example:E6},
we have $R(D_{n,5})=R(E_{6, 1})=R(E_6)$, so $|R(D_{n,5})|=72$. Notice that the
number of roots $Y$ with $(Y\cdot Y)=-2$ is given by $|R(D_5)|$ which
is equal to $40$.

  Now, assume that $n > 0$. Let $Y = \sum_{i=0}^{n}a_i E_i + \sum_{j=1}^5 b_j
F_j$ be a root in $R(D_{n,5})$.
By Proposition \ref{prop-a0}, we have $a_0=0$ or $a_0= \pm 1$ in $Y$.
If $a_0=0$, $Y$ is supported either on $A_n$ or on $D_5$. So we have $|R(A_{n})|+ |R(D_5)|=n(n+1)+40$ roots, in total, for ${    a_0}=0$.

  Assume that $a_0=1$.
Put $D=\sum_{i=1}^n a_i E_i$ and $F=E_0+\sum_{j=1}^5 b_j F_j$.
Then we have
\[-3=(Y\cdot Y)=(D+F)^2=(D\cdot D)+2a_1+(F\cdot F).\] 
{    If $D$ is a positive divisor, then $(D\cdot D)+2a_1 \le (D\cdot D) + a_1^2+1 \le 0$ 
by Lemma \ref{lemma:inequality}. If $D=0$, then $(D\cdot D) + 2a_1=0$. 
Since $(F\cdot F) \le -3$, 
we see that $(D\cdot D)=-2{   a_1}$ ($a_1=0$ or $1$) and $(F\cdot F)=-3$. } 
By the connectedness of ${\rm Supp} (Y)$, 
$D=0$ or $D=\sum_{j=1}^{k}E_j$ for some $1 \le k \le n$, and  $F=E_0$ or $b_1 > 0$.
The number of elements in the set $\{ F \mid F=E_0 \mbox{ or } b_1 > 0 \}$ is equal to
$(72-40)/2=16$ since $|R(D_{0,5})|=72$ and $|R(D_5)|= 40$.
Hence there exist $16(n+1)$ roots with $a_0=1$.
Similarly, there exist $16(n+1)$ roots with $a_0=-1$.
Therefore we have \[|R(D_{n, 5})| = n(n+1)+40+2\cdot 16(n+1) = n^2+33n+72.\] 
\end{proof}

%
\begin{prop}
The number of roots in $R(F_{n})$ is $n^2+55n+126$.

\begin{figure}[h!]
\setlength{\unitlength}{1mm}
\begin{picture}(125,20)(0,15)
\put(30,18){\makebox(0,0){$E_n$}}
\put(30,22){\circle{2}}
\put(40,22){\circle{2}}
\put(40,18){\makebox(0,0){$E_{n-1}$}}
\put(31,22){\line(1,0){8}}
\put(45,22){\makebox(0,0){$\cdots$}}
\put(50,22){\circle{2}}
\put(51,18){\makebox(0,0){$E_1$}}
\put(51,22){\line(1,0){8}}
\put(58.8,20.5){$\blacksquare$}
\put(60,18){\makebox(0,0){$E_0$}}
\put(61,22){\line(1,0){8}}
\put(70,22){\circle{2}}
\put(71,18){\makebox(0,0){$F_1$}}
\put(71,22){\line(1,0){8}}
\put(80,22){\circle{2}}
\put(81,18){\makebox(0,0){$F_2$}}
\put(81,22){\line(1,0){8}}
\put(90,22){\circle{2}}
\put(91,18){\makebox(0,0){$F_3$}}
\put(91,22){\line(1,0){8}}
\put(100,22){\circle{2}}
\put(101,18){\makebox(0,0){$F_4$}}
\put(101,22){\line(1,0){8}}
\put(110,22){\circle{2}}
\put(111,18){\makebox(0,0){$F_5$}}
\put(90,32){\circle{2}}
\put(90,31){\line(0,-1){8}}
\put(94,32){\makebox(0,0){$F_6$}}
\end{picture}
\caption{Triple diagram of type $F_{n}$.}
\label{fig-Fn}
\end{figure}
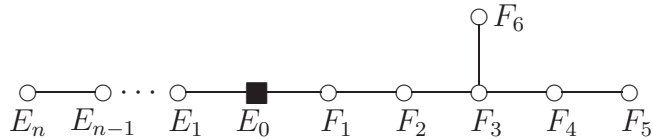
\end{prop}

\begin{proof} The idea is the same as in the preceding
proof. The tree $F_n$ is the glueing of the rational trees
$A_n$ and $E_6$  (see Figure \ref{fig-Fn}). In case $n=0$, we obtain
$|R(F_0)|=|R(E_{7,1})|=|R(E_7)|$ which is equal to $126$.

 Assume that $n > 0$. Let $Y=\sum_{i=0}^n a_i E_i + \sum_{j=1}^6 b_j F_j$ be a root in $R(F_n)$.
By Proposition \ref{prop-a0}, we have $a_0=0$ or $a_0=\pm 1$. When $a_0=0$,
$Y$ is supported either on $A_n$ or on $E_6$, so we have $|R(A_n)| +
|R(E_6)| = n(n+1)+72$ roots.

 Assume that $a_0=1$. Since $\textnormal{Supp}(Y)$ is connected,
$Y$ is of  the form $D+F${   ,} where $D=\sum_{i=1}^{n-1} a_i E_i$ and
$F=E_0+\sum_{j=1}^{6}b_jF_j$.
Consider \[(Y\cdot Y)=(D\cdot D)+{   2a_1}+(F\cdot
F).\] 
As in the proof of the case $D_{n, 5}$ we have
{    $(D\cdot D)=-2a_1$ ($a_1=0$ or $1$)} and $(F\cdot F)=-3$.
This says that $D=0$ or $D= \sum_{i=1}^{k} E_i$ for some $k$ and that
the number of the roots $F$ is $(126-72)/2=27$. 
We also obtain the same number for $a_0 = -1$.
Therefore we have $|R(D_{n,5})|=n^2+55n+126$.\end{proof}

\begin{remark} In Table \ref{tablea}, Section \ref{subsect-lfd}, we provide a general picture of the roots in $R(\Gamma)$ for $A_{m,n,k}$, $B_{m,n}$, $C_{m,n}$, $D_{n,5}$ and $F_n$. 
\end{remark}

%
%
%

\begin{prop}\label{prop:Hn}
The number of roots in $R(H_{n})$ is $(n^3-n)/3$ for $n \ge 5$.

\begin{figure}[h!]
\setlength{\unitlength}{1mm}
\begin{picture}(90,20)(-20,15)
\put(19,18){\makebox(0,0){$E_{n-1}$}}
\put(20,22){\circle{2}}
\put(30,22){\circle{2}}
\put(29,18){\makebox(0,0){$E_{n-2}$}}
\put(21,22){\line(1,0){8}}
\put(31,22){\line(1,0){8}}
\put(41,18){\makebox(0,0){$E_{n-3}$}}
\put(40,22){\circle{2}}
\put(45,22){\makebox(0,0){$\cdots$}}
\put(50,22){\circle{2}}
\put(51,18){\makebox(0,0){$E_4$}}
\put(51,22){\line(1,0){8}}
\put(60,22){\circle{2}}
\put(60,18){\makebox(0,0){$E_3$}}
\put(61,22){\line(1,0){8}}
\put(70,22){\circle{2}}
\put(71,18){\makebox(0,0){$E_2$}}
\put(71,22){\line(1,0){8}}
\put(80,22){\circle{2}}
\put(81,18){\makebox(0,0){$E_1$}}
\put(58.5,31){$\blacksquare$}
\put(60,31){\line(0,-1){8}}
\put(64,32){\makebox(0,0){$E_0$}}
\end{picture}
\caption{Triple diagram of type $H_{n}$.}
\label{fig-Hn}
\end{figure}
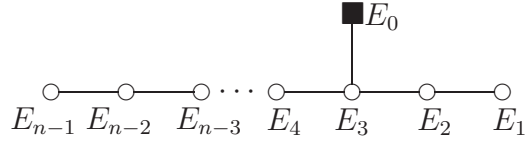
\end{prop}

Proposition \ref{prop:Hn} is a consequence of the following two lemmas.
\begin{lemma}\label{lemma:H5}
Proposition \ref{prop:Hn} is true for $n=5$.
\end{lemma}

\begin{proof}
If $n=5$, then $H_5=D_{5, 5}$ (see Figure \ref{fig-Hn}). The fact that $R(D_{5, 5})=
R(D_{5})$ implies that $|R(H_{5})|=40$. This
completes the proof. \end{proof}

{    

\begin{lemma}\label{prop:Hnnumber}
For each $n \ge 6$, we have $|R(H_{n})| - |R(H_{n-1})| = n^2-n$.
\end{lemma}

\begin{proof} 
We have $R(H_{n}) \supset R(H_{n-1}) := R(H_{n})\cap
\oplus_{i=0}^{n-2}{\mathbb Z}E_{i}$. We only need to show that the
number of the elements in $R(H_{n}) \setminus R(H_{n-1})$ is $n^2-n$.
If $Y = \sum_{i=0}^{n-1}a_i E_i \in R(H_{n}) \setminus R(H_{n-1})$, then $Y$ is one of the
following three types:
\begin{enumerate}
\item[(1)] $Y \in R(A_{n-1})=  R(H_{n})\cap
\oplus_{i=1}^{n-1}{\mathbb Z}E_{i}$ and $a_{n-1} \neq 0$,
\item[(2)] $Y \in R(D_{n-1, 1})=  R(H_{n})\cap
{\mathbb Z}E_{0}\oplus (\oplus_{i=2}^{n-1}{\mathbb Z}E_{i})$,
$a_{0} \neq 0$, and $a_{n-1} \neq 0$,
\item[(3)] the case that $a_{0} \neq 0$, $a_{1} \neq 0$, and $a_{n-1}\neq
  0$.
\end{enumerate}
It is easy to check that the number of all roots of type (1)
is $2(n-1)$ and that the number of all roots of type (2)
is $2(n-2)$. 
By Lemmas \ref{lemma:type3} and \ref{lem-listHn}, 
we see that the number of all roots of type
(3) is $(n-2)(n-3)$.  Hence $|R(H_{n})
\setminus R(H_{n-1})| = 2(n-1)+ 2(n-2)+(n-2)(n-3)=n^2-n$. 
\end{proof} 

The following two lemmas have been used in Lemma \ref{prop:Hnnumber}. 
}

\begin{lemma}\label{lemma:type3}
If $Y = \sum_{i=0}^{n-1}a_i E_i$ is a positive root of type (3) {    in the proof of 
Lemma \ref{prop:Hnnumber},} then
$Y$ satisfies the following conditions:
\begin{enumerate}
\item[(a)] $a_{0}=a_{1}=a_{n-1}=1$,
\item[(b)] $a_{1} \le a_{2} \le a_{3} \ge a_{4} \ge a_{5} \ge \cdots \ge
  a_{n-1}$ and $a_{3} \le 3$,
\item[(c)] $| a_{i} - a_{i+1} | \le 1$ for each $1 \le i \le n-2$.
\end{enumerate}
\end{lemma}

\begin{proof} 
Put $D=E_{1}+E_{2}+\cdots + E_{n-1}$.
Since ${\rm Supp}(Y)$ is connected,
$Y-D>0$. Now we show that $Y-D \in R(H_n)$. Indeed,
by easy calculation we have 
\[(Y-D)^2 = Y^2-2YD+D^2 =
-3+2(a_{1}+a_{n-1}-1)-2=-7+2(a_{1}+a_{n-1}).\]
Here note that $Y^2=-3$ and $a_{0}=1$ by Proposition \ref{prop-a0}.
Moreover, by Lemma \ref{lemma:-2}, we have $(Y-D)^2 \le -2$.
On the other hands, the assumption
$a_{1}>0$ and $a_{n-1}>0$ implies that $(Y-D)^2 \ge -3$.
Since $(Y-D)^2$ is odd, $(Y-D)^2=-3$ and $a_{1}=a_{n-1}=1$.
Hence $Y-D \in R(H_n)$ and 
{    $Y$ satisfies the condition (a).

Note that $Y-D \in R(D_{n-2, 1}) = R(H_n) \cap 
({\mathbb Z}E_{0}\oplus (\oplus_{i=2}^{n-2}{\mathbb Z}E_{i}))$. 
By Example \ref{example:Dn},  $R(D_{n-2, 1}) = R(D_{n-2})$. 
Since $Y-D \in R(D_{n-2})$, 
$Y-D = E_0 + \sum_{i=2}^{n-2} b_i E_i$ 
satisfies 
$b_2 \le 1$, $b_{n-2} \le 1$, $b_2 \le b_3 \ge b_4 \ge b_5 \ge \cdots \ge b_{n-2}$, and 
$| b_i - b_{i+1} | \le 1$ for $2 \le i \le n-3$. 
We easily  check that $Y$ satisfies the conditions
(a), (b), and (c) except $a_{3} \le 3$.
Finally we see that $a_{3} \le a_{2}+1 \le a_{1}+2=3$.
This completes the proof of the lemma. }  
\end{proof}

\begin{lemma}\label{lem-listHn}
The number of all roots of type (3) {    in the proof of 
Lemma \ref{prop:Hnnumber}} is equal to $(n-2)(n-3)$.
\end{lemma}

\begin{proof}
It is easy to see that any cycle satisfying the conditions (a), (b),
and (c) is a root.
Therefore we only need to show that the number of all positive cycles
satisfying (a)-(c) is equal to $(n-2)(n-3)/2$.
For listing up all roots we denote a root by
$(a_2, a_3, a_4, \ldots, a_{n-2})$ since
$a_{0}=a_{1}=a_{n-1}=1$. All roots can be grouped into the following five forms.

\begin{enumerate}
\item $(1, 1, 1, \ldots, 1)$,
\item $(1,\underbrace{2, 2, \ldots, 2}_{i}, \underbrace{1, \ldots, 1}_{j})$, 
\item $(2, \underbrace{2,2, \ldots, 2}_{i}, \underbrace{1, \ldots, 1}_{j})$, 
\item $(2, \underbrace{3, 3, \ldots, 3}_{i},\underbrace{2, \ldots, 2}_{j})$, where {   $i, j\geq 1$} with $i+j=n-4$, and
\item $(2, \underbrace{3, 3, \ldots, 3}_{i}, \underbrace{2, \ldots, 2}_{j}, \underbrace{1, \ldots, 1}_{k})$ where {   $i, j, k\geq 1$} with $i+j+k=n-4$.\\
\end{enumerate}

 We can easily check that the number of roots in the $5$ cases
above are $1, (n-4), (n-4), {   (n-5)}$, and $(n-5)(n-6)/2$, respectively.
Summing up these, we obtain the number of all roots of type (3) which is $(n-2)(n-3)$.
\end{proof}



{    

\begin{proof}[Proof of Proposition \ref{prop:Hn}] 
By Lemmas \ref{lemma:H5} and \ref{prop:Hnnumber}, we easily see that 
$|R(H_{n})| =(n^3-n)/3$ for $n \ge 5$.
\end{proof} 

}


\begin{prop}\label{prop-E71}
The number of roots in $R(E_{7,1})$ is $124$.

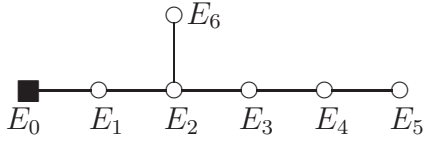
\begin{figure}[h!]
\setlength{\unitlength}{1mm}
\begin{picture}(125,20)(-15,15)
\put(30,18){\makebox(0,0){$E_0$}}
\put(29,20.5){$\blacksquare$}
\put(31,22){\line(1,0){8}}
\put(41,18){\makebox(0,0){$E_1$}}
\put(40,22){\circle{2}}
\put(50,22){\circle{2}}
\put(51,18){\makebox(0,0){$E_2$}}
\put(41,22){\line(1,0){8}}
\put(51,22){\line(1,0){8}}
\put(60,22){\circle{2}}
\put(61,18){\makebox(0,0){$E_3$}}
\put(61,22){\line(1,0){8}}
\put(70,22){\circle{2}}
\put(71,18){\makebox(0,0){$E_4$}}
\put(71,22){\line(1,0){8}}
\put(80,22){\circle{2}}
\put(81,18){\makebox(0,0){$E_5$}}
\put(50,31){\line(0,-1){8}}
\put(50,32){\circle{2}}
\put(54,32){\makebox(0,0){$E_6$}}
\end{picture}
\caption{Triple diagram of type $E_{7,1}$.}
\label{fig-E71}
\end{figure}
\end{prop}

\begin{proof} Note that $R(E_{7}) \supset R(E_{7, 1})$ (see Figure \ref{fig-E71}). Let $Z(E_7)$ be the Artin cycle  of $R(E_7)$.
The coefficient of $E_1$ in $Z(E_7)$ is $2$. For each positive root in $R(E_7)$ except $Z(E_7)$ the coefficient of $E_1$ is less than $2$. Hence $R(E_7) \setminus \{ \pm Z(E_7) \} = R(E_{7, 1})$, and consequently,
 \[|R(E_{7,1})| = |R(E_7)| -2 = 124.\]
\end{proof}

\begin{prop}\label{prop-E81}
The number of roots in $R(E_{8,1})$ is $238$.
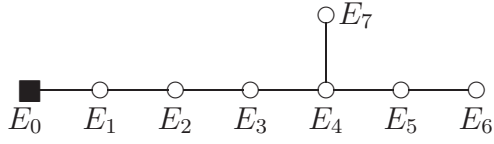
\begin{figure}[h!]
\setlength{\unitlength}{1mm}
\begin{picture}(120,20)(-10,15)
\put(30,18){\makebox(0,0){$E_0$}}
\put(29,20.5){$\blacksquare$}
\put(31,22){\line(1,0){8}}
\put(40,18){\makebox(0,0){$E_1$}}
\put(40,22){\circle{2}}
\put(50,22){\circle{2}}
\put(50,18){\makebox(0,0){$E_2$}}
\put(41,22){\line(1,0){8}}
\put(51,22){\line(1,0){8}}
\put(60,22){\circle{2}}
\put(60,18){\makebox(0,0){$E_3$}}
\put(61,22){\line(1,0){8}}
\put(70,22){\circle{2}}
\put(70,18){\makebox(0,0){$E_4$}}
\put(71,22){\line(1,0){8}}
\put(80,22){\circle{2}}
\put(80,18){\makebox(0,0){$E_5$}}
\put(81,22){\line(1,0){8}}
\put(90,22){\circle{2}}
\put(90,18){\makebox(0,0){$E_6$}}
\put(70,31){\line(0,-1){8}}
\put(70,32){\circle{2}}
\put(74,32){\makebox(0,0){$E_7$}}
\end{picture}
\caption{Triple diagram of type $E_{8,1}$.}
\label{fig-E81}
\end{figure}
\end{prop}

\begin{proof} Note that $R(E_{8}) \supset R(E_{8, 1})$ (see Figure \ref{fig-E81}). Let $Z(E_8)$ be the Artin cycle  of $R(E_8)$.
The coefficient of $E_1$ in $Z(E_8)$ is $2$. For each positive root in $R(E_8)$ except $Z(E_8)$, the coefficient of $E_1$ is less than $2$. Therefore, $R(E_8) \setminus \{ \pm Z(E_8) \} = R(E_{8, 1})$ whence 
\[|R(E_{8,1}) =| R(E_8)| -2 = 238.\]
\end{proof}

\begin{prop}\label{prop-E82}
The number of roots in $R(E_{8,2})$ is $212$.
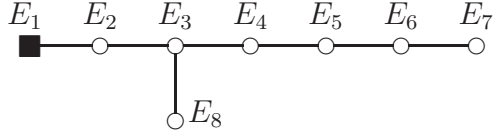
\begin{figure}[h!]
\setlength{\unitlength}{1mm}
\begin{picture}(140,20)(0,20)
\put(40,34){\makebox(0,0){$E_1$}} 
\put(39,28.5){$\blacksquare$}
\put(50,30){\circle{2}}
\put(50,34){\makebox(0,0){$E_2$}}
\put(41,30){\line(1,0){8}} \put(51,30){\line(1,0){8}}
\put(60,30){\circle{2}}
\put(60,34){\makebox(0,0){$E_3$}}
\put(61,30){\line(1,0){8}} \put(70,30){\circle{2}}
\put(70,34){\makebox(0,0){$E_4$}}
\put(71,30){\line(1,0){8}}
\put(80,30){\circle{2}}
\put(80,34){\makebox(0,0){$E_5$}}
\put(81,30){\line(1,0){8}}
\put(90,30){\circle{2}}
\put(90,34){\makebox(0,0){$E_6$}}
\put(91,30){\line(1,0){8}}
\put(100,30){\circle{2}}
\put(100,34){\makebox(0,0){$E_7$}}
\put(60,29){\line(0,-1){8}}
\put(60,20){\circle{2}}
\put(64,21){\makebox(0,0){$E_8$}}
\end{picture}
\caption{Triple diagram of type $E_{8,2}$.}
\label{fig-E82}
\end{figure}
\end{prop}

\begin{proof} Note that $R(E_8)\supset R(E_{8,2})$ (see Figure \ref{fig-E82}). Let us denote a root $\sum a_iE_i$ in $R(E_8)$ by $(a_1,a_2,\ldots,a_8)$. The following positive roots in $R(E_8)$ are not contained in $R(E_{8,2})$.
\begin{eqnarray*}&&(2, 3, 4, 3, 2, 1, 0, 2), (2, 3, 4, 3, 2, 1, 1, 2), (2,3, 4, 3, 2, 2, 1, 2), 
\\ && (2, 3, 4, 3, 3, 2, 1, 2),  (2, 3, 4, 4, 3,2, 1, 2), (2, 3, 5, 4, 3, 2, 1, 2), 
\\ &&(2, 3, 5, 4, 3, 2, 1, 3), (2, 4,5, 4, 3, 2, 1, 2),  (2, 4, 5, 4, 3, 2, 1, 3),
\\ &&(2, 4, 6, 4, 3, 2, 1, 3), (2, 4, 6, 5, 3, 2, 1, 3), (2, 4, 6, 5, 4, 2, 1, 3), 
\\&& (2, 4, 6, 5, 4, 3, 1, 3), (2, 4, 6, 5, 4, 3, 2, 3). \end{eqnarray*}

 We have $14$ positive roots, in total $28$  which are not contained in $R(E_{8,2})$. Hence, $|R(E_{8,2})|=|R(E_8)|-28=212$.
\end{proof}

For each triple {   Dynkin diagram}, we can easily verify the following statements.

\begin{theorem}
{    Let $\Gamma_0$ be a Dynkin diagram of type $A, D$ or $E$. Let $\Gamma$ be the triple Dynkin diagram obtained by replacing a vertex $E_i$ of $\Gamma_0$ by a vertex with weight $3$. Then there exists the highest root
$Z(\Gamma)=\sum a_jE_j \in R(\Gamma_0)$ among the roots in $R(\Gamma_0)$ with $a_i =1$.  
Furthermore, $Z(\Gamma)$ is the Artin cycle of $\Gamma$ and $R(\Gamma) = \{ Y
\in R(\Gamma_0) \mid -Z(\Gamma) \le Y \le Z(\Gamma) \}$.  }
\end{theorem}

%
%

\section{Linear free divisors}\label{sect-free}

A reduced hypersurface $D\subset (\mathbb{C}^n,0)$ is called a free divisor if the module $\textnormal{Der}(- \textnormal{log }D)$ of logarithmic vector fields along $D$ is a locally free module of rank $n$ over $\mathcal{O}_{\mathbb{C}^n,0}$. By Saito's criterion (\cite{saito}), $D$ is a free divisor if and only if there exists a basis $\chi_1,\ldots, \chi_n$ of $\textnormal{Der}(- \textnormal{log }D)$ such that the determinant of the matrix formed by the coefficients of $\chi_i$ is a reduced equation defining $D$. If, in particular, each $\chi_i$ is a weight zero vector field, i.e. of the form $\sum_{i,j} \xi_{ij}x_j \frac{\partial}{\partial x_i}$, for some $\xi_{ij}\in \mathbb{C}$, then $D$ is a \textit{linear free divisor}. In this section, we will recall the examples of linear free divisors arising in representation theory.

  Let $\alpha\colon \textnormal{GL}_n(\mathbb{C}) \times \mathbb{C}^n\rightarrow \mathbb{C}^n$ be a group action given by the right multiplication. Consider the restriction
\[\alpha_x\colon \textnormal{GL}_n(\mathbb{C}) \times \{x\} \rightarrow \mathbb{C}^n\]
for some $x\in \mathbb{C}^n$. Then each element $v\in\mathfrak{gl}_n$ gives rise to a vector field 
{   $\chi_v$} on $\mathbb{C}^n$ defined by 
\begin{equation}\label{eq-inf}\chi_v(x)=d_e\alpha_x(v).\end{equation}
In particular, the elementary matrix $\mathfrak{e}_{ij}$, which has $1$ in the $(i,j)$-th entry, $0$ everywhere else, corresponds to the vector field $x_j\frac{\partial}{\partial x_i }$.

  More generally,  let $G\subseteq \textnormal{GL}_n(\mathbb{C}) $ be a connected algebraic subgroup of dimension $n$. 
Let $\chi_1,\ldots, \chi_n$ be vector fields generating the infinitesimal action of $G$ induced by (\ref{eq-inf}). Then, each $\chi_i$ is of the form $\sum_{i,j} \xi_{ij}x_j \frac{\partial}{\partial x_i}$, for some ${   \xi_{ij} } \in \mathbb{C}$, and corresponds to the matrix  $\sum_{i,j} \xi_{ij}\mathfrak{e}_{ij}$. Let $\Delta$ be the matrix of the coefficients $\{\chi_1,\ldots, \chi_n\}$ with respect to the standard basis $\frac{\partial}{\partial x_1}, \ldots, \frac{\partial}{\partial x_n}$ of $\textnormal{Der}_{\mathbb{C}^n}$. Then, the determinant of $\Delta$ defines the \textit{discriminant} $D$ which consists of points $x$ where $\chi_1,\ldots, \chi_n$ fail to span the tangent space $T_x\mathbb{C}^n$. Furthermore, $D$ is a homogeneous divisor in $\mathbb{C}^n$ of degree $n$. By Saito's criterion, if the determinant of $\Delta$ is reduced then $D$ is a free divisor (necessarily linear) and $\chi_1,\ldots, \chi_n$ form a basis of $\textnormal{Der}(- \textnormal{log }D)$ (\cite[Lemma 2.4]{granger-mond-r-s}).

 Conversely, assume that $D\subset \mathbb{C}^n$ is a linear free divisor and consider the group
\[G_D:=\{A\in \textnormal{GL}_n(\mathbb{C}) \mid A(D)=D\}.\]
Let $G_D^0$ be the {   $n$-dimensional} connected component of $G_D$ containing the identity element. Then {   $G_D^0$ acts on $\mathbb{C}^n$ with a single open orbit and the open orbit corresponds to the complement $\mathbb{C}^n\setminus D$} (\cite[Lemma 2.3]{granger-mond-r-s}).

 In the following {   sub}section, 
 we study examples of free divisors arising in quiver representations.

\subsection{{   Linear free divisors arising from quiver representations}}\label{sect-quiv}

A \textit{quiver} $Q$ is an oriented graph together with sets $Q_0$ and $Q_1$ consisting of the vertices and arrows, respectively. A representation $V$ of a quiver over a field $k$ consists of a vector space $V_v$ of dimension $a_v$ for
each vertex $v\in Q_0$ and a $k$-linear map $V(\alpha)\colon V_{t\alpha}\rightarrow V_{h\alpha}$ for each arrow $\alpha\in Q_1$, where $t\alpha$ is the start and $h\alpha$ is the end of the arrow $\alpha$.

 Let ${\bf a}:=\left(a_v\right)_{v\in Q_0}\in \mathbb{N}^{\mid Q_0 \mid}$ be the dimension vector assigned to $Q$. Then, the $k$-vector space of representations of $Q$ is defined by
\[\textnormal{Rep}(Q ,{\bf a}):=\prod_{\alpha\in Q_1 } \textnormal{Hom}_k\left(V_{t\alpha},V_{h\alpha}\right)\cong \prod_{\alpha\in Q_1 }
\textnormal{Hom}_k\left(k^{a(t\alpha)},k^{a(h\alpha)}\right).\]

The group $\textnormal{GL}(Q ,{\bf a}):=\prod_{v \in Q} \textnormal{GL}_{a_v}(k)$ acts on $\textnormal{Rep}(Q ,{\bf a})$ by 
\begin{equation}\label{g-act}((g_v)_{v\in Q_0}, V)_{\alpha \in Q_1}\mapsto (g_{h\alpha}\cdot A_\alpha\cdot g_{t\alpha}^{-1})_{\alpha\in Q_1}\end{equation}
where $(g_v)_{v\in Q_0}\in \textnormal{GL}(Q ,{\bf a})$. 

 A \textit{morphism} $\phi\colon V \rightarrow W$ of representations is well-defined if there exists a commutative diagram 
\[\xymatrix{ V_{t\alpha} \ar[r]^{V(\alpha)}  \ar[d]_{\phi_{t\alpha}} & V_{h\alpha} \ar[d]^{\phi_{h\alpha}}\\
W_{t\alpha} \ar[r]^{W(\alpha)} & W_{h\alpha}}\]
where $k$-linear maps $\phi_{t\alpha}$ and $\phi_{h\alpha}$, for each $\alpha\in Q_1$. Moreover, $\phi $ is an isomorphism if $\phi_{v}$ is an isomorphism for all $v\in Q_0$. On the other hand, the \textit{direct sum} of two representations $V\in \textnormal{Rep}(Q ,{\bf a})$ and $W\in \textnormal{Rep}(Q ,{\bf b})$ is $V\oplus W\in \textnormal{Rep}(Q ,{\bf a}+{\bf b})$ with $(V\oplus W)_v:=V_v\oplus W_v$ and 
\[(V\oplus W)(\alpha):=\begin{bmatrix}V_\alpha & 0 \\ 0 & W_\alpha \end{bmatrix}\]
for all $v\in Q_0$ and $\alpha\in Q_1$.

\begin{defn}[\cite{buch-mond}] A representation $V'\in \textnormal{Rep}(Q ,{\bf a}')$ is \textit{decomposable} if it is isomorphic to the direct sum of two nontrivial representations, that is, $V'=V\oplus W$ for $V\in \textnormal{Rep}(Q ,{\bf a})$ and $W\in \textnormal{Rep}(Q,{\bf b})$ and ${\bf a}'={\bf a}+{\bf b}$. Otherwise, $V'$ is called indecomposable.
A quiver is of \textit{finite representation type} if it has only finitely many indecomposable representations, up to isomorphism.
\end{defn}

 The \textit{Tits form} of a dimension vector $\textbf{a}$ is given by
\begin{eqnarray*}\langle  \textbf{a},\textbf{a} \rangle &:=& \sum_{v\in Q_0} a_v^2-\sum_{\alpha\in Q_1}a_{t\alpha}a_{h\alpha} \\ &=& 
{    \sum_{v\in Q_0}}\textnormal{dim}_k \textnormal{Hom}_k (V_v,V_v)-{   \sum_{\alpha\in Q_1}}\textnormal{ dim}_k \textnormal{Hom}_k(V_{t\alpha},V_{h\alpha}) \\
&{   =}& {    \textnormal{dim}_k \textnormal{GL}(Q ,{\bf a}) - \textnormal{ dim}_k \textnormal{Rep}(Q ,{\bf a})}.
\end{eqnarray*}

\begin{defn}[Definition 3.2, \cite{buch-mond}] {   The dimension vector $\textbf{a}$ is called a root if $ \textnormal{Rep}(Q ,{\bf a})$ contains an indecomposable representation.} A root  $\bf a$ is called  \textit{sincere} if  $a_i\geq 1$ for all $i$. A root is \textit{real} if $\textnormal{Rep}(Q ,{\bf a})$ contains exactly one orbit of, necessarily isomorphic, indecomposable representations. If a general representation in $\textnormal{Rep}(Q ,{\bf a})$ is indecomposable then $\textbf{a}$ is a \textit{Schur} root. 
\end{defn}

{    The role of $G^0_D$ described in the beginning of Section \ref{sect-free} for the case of quiver representations is played by the quotient $\textnormal{PGL}(Q ,{\bf a}):=\textnormal{GL}(Q ,{\bf a})/Z_0$ where  $Z_0:=\mathbb{C}^*\textnormal{id}\subset \textnormal{GL}(Q ,{\bf a})$ is the 1-dimensional central subgroup acting trivially on $\textnormal{Rep}(Q ,{\bf a})$ under the action given by (\ref{g-act}). And $\textnormal{Rep}(Q ,{\bf a})$ takes the role of $\mathbb{C}^n$ in 
the beginning of Section \ref{sect-free}. Let us assume that $Q$ is a quiver whose underlying graph is a tree and $\textbf{a}$ is a sincere root. We have 
\[\textnormal{dim}_k \textnormal{PGL}(Q ,{\bf a}) = \textnormal{dim}_k \textnormal{Rep}(Q ,{\bf a})\]
if and only if $\langle  \textbf{a},\textbf{a} \rangle =1$, 
and the latter holds if and only if $\textbf{a}$ is a real root (\cite{kac}, see also \cite[Proposition 3.7]{buch-mond}). 
Under the condition that $\langle  \textbf{a},\textbf{a} \rangle =1$, 
let $\Delta$ be the matrix defined in the beginning of Section \ref{sect-free}. 
Let $D$ be the discriminant defined by $\det \Delta$.  
Then $\textnormal{Rep}(Q ,{\bf a}) \setminus D$ is a single open orbit of 
$\textnormal{PGL}(Q ,{\bf a})$ if 
$\textbf{a}$ is a real Schur root (\cite[Section~4]{granger-mond-r-s}). If, in addition to these conditions, 
$\det \Delta$ is reduced, then $D$ is a linear free divisor (\cite[Lemma~2.4]{granger-mond-r-s}). Conversely, if $(Q,\mathbf{a})$ defines a linear free divisor then $\mathbf{a}$ is a real Schur root (\cite[Theorem 3.4]{granger-mond-schulze}).}

{   In the case of Dynkin quivers, we have the following result of \cite{buch-mond}.}

\begin{theorem}[Corollary 5.5, \cite{buch-mond}]\label{thm-buch-mond} Let $Q$ be a Dynkin quiver and $\mathbf{a}$ be a real Schur root. Then the discriminant $D$ {   of} the action of $\textnormal{PGL}(Q ,{\bf a})$ on  $\textnormal{Rep}(Q ,{\bf a})$ is a linear free divisor. 
\end{theorem}

In the next {   sub}section, we will show that linear free divisors also arise as discriminants in rational triple quiver representations; hence the roots {    in triple root systems} 
are also {   real} Schur roots. 

Let $\Gamma$ denote one of the trees studied in Proposition{   s} 
\ref{prop:anmk}-\ref{prop:Hn} {   and \ref{prop-E71}-\ref{prop-E82}}.  
\begin{prop}\label{prop-tit} {    For an RTP-tree $\Gamma$}, $\langle  \textbf{a},\textbf{a} \rangle =1$ {   if} 
$Y\in R(\Gamma)$ where $Y=\sum a_iE_i$ with ${\bf a}=(a_0,a_1,\ldots ,a_n)$. 
\end{prop}

\begin{proof} It follows from definitions of the Tits form and the intersection matrix associated to the rational triple quiver. 
\end{proof}

We will denote an RTP-tree $\Gamma$ by $Q$ after assigning an orientation and 
call it a \textit{rational triple quiver}. From now on, we will use the term root for a dimension vector $\bf a$ assigned to a rational triple quiver and (or equivalently for a divisor $Y=\sum a_iE_i$ in $R(\Gamma )$). 

By Proposition \ref{prop-subtree}, any subquiver of a rational triple quiver is a rational triple quiver or a Dynkin quiver. So, each root is a real root by {   the results mentioned above}. In the next {   sub}section, we will show that linear free divisors also arise as discriminants in rational triple quiver representations; hence the roots {    in triple root systems} are also {   real} Schur roots.

\subsection{Rational triple quivers and linear free divisors}\label{subsect-lfd}

A vertex $v\in Q_0$ is called a \textit{source} (resp. \textit{sink}) if there is no arrow ending (resp. starting) at $v$. {   Let  $\mathbf{a}=(a_0,a_1,\ldots,a_n)$ be 
a dimension vector.} 


\begin{defn}[\cite{bern-gel-pono},\cite{granger-mond-schulze}] Let $v\in Q_0$ be a source with $a_v\leq \sum_{v\rightarrow v_i\in Q_1} a_{v_i}$. The \textit{reflection functor} with respect to $v$ is a transformation which produces a quiver $Q^*$  with $a^*_v=\sum_{v\rightarrow v_i\in Q_1} a_{v_i}-a_v$ and $a_{v_i}^*=a_{v_i}$ for $v_i\neq v$, all arrows involving $v$ are reversed. So, $v^*$ is a sink and $a_v^*\geq \sum_{v^*\rightarrow v_i^*\in Q_1^*} a_{v_i}^*$. 

 This correspondence can be extended to a correspondence between open subsets $\textnormal{Rep}'(Q ,{\bf a})\subset \textnormal{Rep}(Q ,{\bf a})$ and $\textnormal{Rep}'(Q^* ,{\bf a}^*) \subset \textnormal{Rep}(Q^* ,{\bf a}^*)$ {    defined in 
 \cite[Theorem~3.26]{granger-mond-schulze} } 
 as follows. Let $V\in \textnormal{Rep}'(Q ,{\bf a})$  and denote its image under a reflection by $V^*\in  \textnormal{Rep}'(Q ,{\bf a}^*)$. Then $V^*_v$ is defined be the cokernel of the map 
\[ (f_{v\rightarrow v_i})_{ v\rightarrow v_i\in Q_1} \colon V_v \rightarrow \bigoplus_{v\rightarrow v_i\in Q_1} V_{v_i} \] 
and $V_{v_i}=V_{v_i}^*$ for $v_i\neq v$. 

 A similar definition exists in the case of a sink. If $v\in Q_0$ is a sink with $a_v\geq \sum_{{    v\leftarrow v_i\in Q_1}} a_{v_i}$, $Q^*$ is obtained by reversing all arrows involving $v$ and setting $\mathbf{a}^*$ as above. The corresponding vertex $v^*$ is a source and $a_v^*\leq \sum_{v^*\rightarrow v_i^*\in Q_1^*} a_{v_i}^*$. In this case, $V^*\in  \textnormal{Rep}'(Q^*,{\bf a}^*)$ is obtained by setting $V^*_v$ as the kernel of the map 
\[ (f_{v\leftarrow v_i})_{ v\leftarrow v_i\in Q_1} \colon \bigoplus_{v\rightarrow v_i\in Q_1} V_{v_i}  \rightarrow V_v\] 
and $V_{v_i}=V_{v_i}^*$ for $v_i\neq v$.
 \end{defn}

  It was shown by Sato and Kimura that reflection functors (or \textit{castling transformations} as they call them) give a one-to-one correspondence between relative invariants of representations. Moreover, if $D$ is a linear free divisor coming from a quiver representation $V$ then the discriminant $D^*$ of $V^*$ is also a linear free divisor {   (\cite[Theorem 3.26]{granger-mond-schulze}).} Using this result and Theorem \ref{thm-buch-mond}, we will prove that if $Q$ is a rational triple quiver 
  {   and ${\bf a}$ is a positive root  in the corresponding triple root system $R(\Gamma)$, } 
 then its representation {    space $\textnormal{Rep}(Q ,{\bf a})$} also yields a linear free divisor. First we prove the following lemma to simplify our claims.

\begin{lemma}\label{lem-An} Let $Q$ be a quiver of type $A_n$ (see Figure \ref{fig-An}) with the root $\mathbf{a}=(1,1,\ldots,1)\in\mathbb{N}^n$ and any chosen orientation. Then any {    general} representation $V\in \textnormal{Rep}(Q ,{\bf a})$ can be transformed into the trivial representation of the subquiver
\begin{figure}[h!]
\setlength{\unitlength}{1mm}
\begin{picture}(120,10)(0,25)
\put(40,26){\makebox(0,0){$0$}}
\put(40,33){\makebox(0,0){$v_1$}}
\put(40,30){\circle{2}}
\put(50,30){\circle{2}}
\put(50,26){\makebox(0,0){$0$}}
\put(50,33){\makebox(0,0){$v_2$}}
\put(41,30){\line(1,0){8}}
\put(51,30){\line(1,0){8}}
\put(60,30){\circle{2}}
\put(60,26){\makebox(0,0){$0$}}
\put(60,33){\makebox(0,0){$v_3$}}
\put(65,30){\makebox(0,0){$\cdots$}}
\put(70,30){\circle{2}}
\put(70,26){\makebox(0,0){$0$}}
\put(71,33){\makebox(0,0){$v_{n-1}$}}
\put(71,30){\line(1,0){8}}
\put(80,30){\circle{2}}
\put(80,26){\makebox(0,0){$1$}}
\put(80,33){\makebox(0,0){$v_n$}}
\end{picture}
\end{figure}
\newpage \noindent by successive reflection functors; in other words, into $\{0\}= \textnormal{Rep}(Q^* ,{\bf a}^*)$ where $Q^*$ is the quiver consisting of just one vertex $v_n$ with $\mathbf{a}^*=1$.
\end{lemma}

\begin{proof}  For any given orientation, the vertex $v_1$ at the left end is either a sink or a source. In either case, by applying the appropriate reflection with respect to $v_1$, we get a subrepresentation $V^*$ of $V$ with  $V^*_{v_1}=0$, $V^*_{v_i}=V_{v_i}$ and $a_{v_1}=0$, $a_{v_i}=1$ for all $i=2,\ldots, n$. Therefore, the claim easily follows by applying reflection functors at the vertices $v_1$, $v_2$ and $v_{n-1}$ in the given order. 
\end{proof}

\begin{example}\label{ex-Q-Bmn} Let us consider the rational triple quiver of type $B_{m,n}$ with any given orientation and the root $\mathbf{a}=(a_1,\ldots, a_n,1,\ldots,1)$. Then the discriminant in $\textnormal{Rep}(Q ,{\bf a})$ is a linear free divisor. See Figure \ref{fig-Q-Bmn} for the quiver with an {   arbitrary} orientation.

\begin{figure}[h!]
\setlength{\unitlength}{1mm}
\begin{picture}(125,15)(-6,15)
\put(30,18){\makebox(0,0){$1$}}
\put(30,22){\circle{2}}
\put(31,22){\vector(1,0){5}}
\put(40,22){\circle{2}}
\put(35,25){\makebox(0,0){$A_m$}}
\put(40,18){\makebox(0,0){$1$}}
\put(31,22){\line(1,0){8}}
\put(45,22){\makebox(0,0){$\cdots$}}
\put(50,22){\circle{2}}
\put(55,25){\makebox(0,0){$A_1$}}
\put(57,22){\vector(-1,0){3}}
\put(50,18){\makebox(0,0){$1$}}
\put(51,22){\line(1,0){8}}
\put(58.8,20.5){$\blacksquare$}
\put(61,22){\vector(1,0){5}}
\put(60,18){\makebox(0,0){$1$}}
\put(61,22){\line(1,0){8}}
\put(70,22){\circle{2}}
\put(70,18){\makebox(0,0){$a_2$}}
\put(71,22){\vector(1,0){5}}
\put(71,22){\line(1,0){8}}
\put(80,22){\circle{2}}
\put(80,18){\makebox(0,0){$ a_3$}}
\put(97,22){\vector(-1,0){3}}
\put(85,22){\makebox(0,0){$\cdots$}}
\put(90,22){\circle{2}}
\put(90,18){\makebox(0,0){$a_{n-1}$}}
\put(91,22){\line(1,0){8}}
\put(100,22){\circle{2}}
\put(100,18){\makebox(0,0){$ a_n$}}
\put(70,32){\circle{2}}
\put(70,31){\line(0,-1){8}}
\put(74,32){\makebox(0,0){$a_1$}}
\put(70,23){\vector(0,1){5}}
\end{picture}
\caption{A representation of a rational triple quiver of type $B_{m,n}$.}
\label{fig-Q-Bmn}
\end{figure}

{    As in the proof of Lemma \ref{lem-An},} $Q$ can transformed into a Dynkin quiver $Q^*$  of type $D_{n+1}$ with the root $\mathbf{a}^*=(a_1,\ldots,a_n, {   1}, 0,\ldots,0)$. For example, the quiver in 
Figure \ref{fig-Q-Bmn} is transformed into the quiver in Figure \ref{fig-Q*-Bmn}.

\begin{figure}[h!]
\setlength{\unitlength}{1mm}
\begin{picture}(125,20)(-6,15)
\put(30,18){\makebox(0,0){$0$}}
\put(30,22){\circle{2}}
\put(40,22){\circle{2}}
\put(40,18){\makebox(0,0){$0$}}
\put(31,22){\line(1,0){8}}
\put(45,22){\makebox(0,0){$\cdots$}}
\put(50,22){\circle{2}}
\put(50,18){\makebox(0,0){$0$}}
\put(51,22){\line(1,0){8}}
\put(58.8,20.5){$\blacksquare$}
\put(60,18){\makebox(0,0){$1$}}
\put(61,22){\line(1,0){8}}
\put(61,22){\vector(1,0){5}}
\put(70,22){\circle{2}}
\put(70,18){\makebox(0,0){$a_2$}}
\put(71,22){\vector(1,0){5}}
\put(71,22){\line(1,0){8}}
\put(80,22){\circle{2}}
\put(80,18){\makebox(0,0){$ a_3$}}
\put(97,22){\vector(-1,0){3}}
\put(85,22){\makebox(0,0){$\cdots$}}
\put(90,22){\circle{2}}
\put(90,18){\makebox(0,0){$a_{n-1}$}}
\put(91,22){\line(1,0){8}}
\put(100,22){\circle{2}}
\put(100,18){\makebox(0,0){$ a_n$}}
\put(70,32){\circle{2}}
\put(70,31){\line(0,-1){8}}
\put(74,32){\makebox(0,0){$a_1$}}
\put(70,23){\vector(0,1){5}}
\end{picture}
\caption{A representation of a rational triple quiver of type $B_{m,n}$.}
\label{fig-Q*-Bmn}
\end{figure}

 An easy calculation shows that $\langle  \textbf{a},\textbf{a} \rangle=1$ if and only if $\langle  \textbf{a}^*,\textbf{a}^* \rangle=1$; in other words, $\mathbf{a}$ is a root if and only if $\mathbf{a}^*$ is. Therefore, the discriminant $D$ in  $\textnormal{Rep}(Q ,{\bf a})$ is a linear free divisor since $D^*$ in $\textnormal{Rep}(Q^*,{\bf a}^*)$ is a linear free divisor  (cf. Theorem \ref{thm-buch-mond} and \cite[Proposition 3.26]{granger-mond-schulze}). Moreover, we have
\[\textnormal{Rep}(Q ,{\bf a})\cong \textnormal{Rep}(Q^* ,{\bf a}^*)\times \mathbb{C}^m,\]
and if $D^*=V(f)$ then $D=V(A_1A_2\cdots A_mf)$ where 
{    $A_i$ is the $i$-th projection 
$\textnormal{Rep}(Q ,{\bf a})\cong \textnormal{Rep}(Q^* ,{\bf a}^*)\times 
\mathbb{C}^m\rightarrow \mathbb{C}^m \rightarrow \mathbb{C}$ 
for $1 \le i  \le m$.} So, the divisors $D$ and $D^*$ are related as follows
\begin{equation}\label{eq-D-D*} D=(D^*\times \mathbb{C}^m)\bigcup \left(\bigcup_{i=1}^m \mathbb{C}^{\textnormal{dim}D^*+1}\times V(A_i)\right), \end{equation} 
{    where $V(A_i) := \{ (z_1, \ldots, z_m) \in \mathbb{C}^m \mid z_i=0 \}$ }
(cf. \cite[Proposition 3.29]{granger-mond-schulze}). 
\end{example}

\begin{theorem}[Theorem 3.9, \cite{granger-mond-schulze}]\label{thm-tree} Let $Q$ be any quiver and $\mathbf{a}$ a root associated with $Q$. If the discriminant in the representation space of $Q$ is a linear free divisor then $Q$ is a tree.
\end{theorem}

 In the following theorem, we present a new case where the converse of Theorem \ref{thm-tree} is true.

\begin{theorem}\label{prop-lfd} Let $Q$ be a rational triple quiver and $\mathbf{a}$ be a sincere triple root {    in the corresponding triple root system $R(\Gamma)$.} Then, the discriminant in the representation space $\textnormal{Rep}(Q,\mathbf{a})$ is a linear free divisor independently of the orientation. 
\end{theorem}

\begin{proof} We notice that the underlying graphs $E_{7,1}$, $E_{8,1}$ and $E_{8,2}$ are of type $E_7$, $E_8$ and $E_8$, respectively. Since $\textbf{a}$ is a root, the claim follows from Theorem \ref{thm-buch-mond}. 

 As for the series  $A_{n,m,k}$, $B_{m,n}$, $C_{m,n}$, $D_{n,5}$ and $F_{n}$ (see Figures \ref{fig-Anmk}-\ref{fig-Fn}), the general forms of the roots are indicated in Table \ref{tablea}. Each of them necessarily consists of a subquiver of type $A_{n'}$ with $1$ at its vertices for some $n'$ (cf. the results in Section \ref{sect-tripleroots}). In other words, $\textbf{a}=\textbf{a}_1\oplus \textbf{a}_2$ for $\textbf{a}_1=(1,\ldots, 1)\in \mathbb{N}^{n'}$ and $\textbf{a}_2\in \mathbb{N}^{\ell}$ for some $\ell\geq 3$. We have already observed that a rational triple quiver of type $B_{m,n}$ with such dimension vector yields a linear free divisor in Example \ref{ex-Q-Bmn}. Similarly, by Lemma \ref{lem-An}, the representations of $A_{m,n,k}$, $C_{m,n}$, $D_{n,5}$ and $F_{n}$ can be transformed into the representations of the Dynkin quivers of types {   $A_{n+k+1}$}, $D_m$, $E_6$ and $E_7$, in the given order. (In fact, the linear free divisors arising from them are related by equations similar to (\ref{eq-D-D*}).)  See Table \ref{tablea} for a summary in which $\underline{1}$ represents the vertex of weight $-3$.  Moreover,  $\textbf{a}$ is a root if and only if $\textbf{a}_2$ is a root. Therefore, in each case the discriminant is linear free divisor by Theorem \ref{thm-buch-mond} and \cite[Proposition 3.26]{granger-mond-schulze}. 

\begin{table}[ht!]
\caption[]{Rational triple quivers}
\label{tablea}
\begin{center}
\begin{tabular}{|m{0.875cm}|m{4.85cm}| m{4.85cm}|m{1.1cm}|} \hline
\multirow{2}{*}{Label} & \multirow{2}{*}{\makebox[4cm][c]{$\mathbf{a}$}} & \multirow{2}{*}{\makebox[4.6cm][c]{$\mathbf{a}^*$}} & Dynkin \\
& & & type  \\
\hline
$A_{m,n,k}$ & $\xymatrixcolsep{3pt}\xymatrixrowsep{-2pt}\xymatrix{ & 1 & ... &   1 & \underline{1} & 1 & ... & 1   \\ & & & & & & & \\&  &  & &  1   &  & & \\ &  & & & & & & \\& & & & . & & & \\ && &  &  . && &  \\ && & & . & & & \\ & & & & & & & \\  && & & 1 &  & &}$  
& 
$\xymatrixcolsep{3pt}\xymatrixrowsep{-2pt}\xymatrix{ 0 & ... &   0 & \underline{1} & 1 & ... & 1   \\ & & & & & & \\ &  & &  1   &  & & \\   & & & & & & \\ & & & . & & & \\ & &  &  . && &  \\ & & & . & & & \\  & & & & & & \\  & & & 1 &  & &}$  & ${   A_{n+k+1}}$
 \\\hline 
$B_{m,n}$ & $\xymatrixcolsep{2pt}\xymatrixrowsep{0pt}\xymatrix{ & & & & a_1 & & & & \\   1 & ... & 1 & \underline{1} & a_2 & a_3 & ... & a_n}$
& 
$\xymatrixcolsep{2pt}\xymatrixrowsep{0pt}\xymatrix{ & & & & a_1 & & & & \\   0 & ... & 0 & \underline{1} & a_2 & a_3 & ... & a_n }$ & $D_{n+1}$ 
  \\\hline
$C_{m,n}$ & $\xymatrixcolsep{1.5pt}\xymatrixrowsep{0pt}\xymatrix{ & & & &  & & a_1\hskip12pt  &\\ 1 & ...  & 1 & \underline{1} & a_2  & ...  & a_{m-1} & a_{m}}$ 
& $\xymatrixcolsep{1.5pt}\xymatrixrowsep{0pt}\xymatrix{ & & & &  & & a_1\hskip12pt &\\ 0 & ...  & 0 & \underline{1} & a_2 & ...  &a_{m-1} & a_m}$ & $D_{m+1}$
 \\\hline
$D_{n,5}$ & $\xymatrixcolsep{2pt}\xymatrixrowsep{0pt}\xymatrix{ & & & & & a_1 & & \\1 & ... & 1 & \underline{1} & a_2 & a_3 & a_4 & a_5}$ 
& $\xymatrixcolsep{2pt}\xymatrixrowsep{0pt}\xymatrix{ & & & & & a_1 & & \\0 & ... & 0 & \underline{1} & a_2 & a_3 & a_4 & a_5}$ & $E_6$
\\\hline
$F_n$ & $\xymatrixcolsep{1.5pt}\xymatrixrowsep{0pt}\xymatrix{ & && & & & a_1 & &\\ 1 & ... & 1 & \underline {1} & a_2 & a_3 & a_4 & a_5 & a_6}$ 
&
$\xymatrixcolsep{1.5pt}\xymatrixrowsep{0pt}\xymatrix{ & && & & & a_1 & &\\ 0 & ... & 0 & \underline {1} & a_2 & a_3 & a_4 & a_5 & a_6}$ 
& $E_7$
 \\\hline
\end{tabular}
\end{center}
\end{table}

  Now, we prove the claim for $H_n$ explicitly by considering the sincere roots which are listed in the proof of Lemma \ref{lem-listHn}. Note that we do not indicate the directions of the arrows in the figures as our claim is independent of the given orientation. 

\begin{enum}
\item[1.] $\mathbf{a}=(1,1,\ldots,1)\in\mathbb{N}^n$. A quiver representation in $\textnormal{Rep}(Q ,{\bf a})$ is given in Figure \ref{fig-Q-Hn1} where $Q_1=\{A_0,A_1,\ldots,A_{n-1}\}$. It is easy to see that the discriminant is the normal crossing divisor
\[D=\{A_0A_1\cdots A_{n-1}=0\}\subset \textnormal{Rep}(Q ,{\bf a})\cong \mathbb{C}^{n}.\]
Hence, it is a linear free divisor.

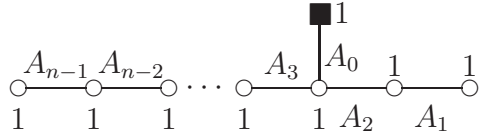
\begin{figure}[h!]
\setlength{\unitlength}{1mm}
\begin{picture}(125,20)(5,15)
\put(40,22){\circle{2}}
\put(45,25){\makebox(0,0){$A_{n-1}$}}
\put(40,18){\makebox(0,0){$1$}}
\put(45,25){\makebox(0,0)}
\put(41,22){\line(1,0){8}}
\put(50,18){\makebox(0,0){$1$}}
\put(55,25){\makebox(0,0){$A_{n-2}$}}
\put(55,25){\makebox(0,0)}
\put(50,22){\circle{2}}
\put(51,22){\line(1,0){8}}
\put(60,18){\makebox(0,0){$1$}}
\put(60,22){\circle{2}}
\put(65,22){\makebox(0,0){$\cdots$}}
\put(70,22){\circle{2}}
\put(70,18){\makebox(0,0){$1$}}
\put(75,25){\makebox(0,0){$A_{3}$}}
\put(75,25){\makebox(0,0)}
\put(71,22){\line(1,0){8}}
\put(80,22){\circle{2}}
\put(80,18){\makebox(0,0){$1$}}
\put(85,18){\makebox(0,0){$A_{2}$}}
\put(85,19){\makebox(0,0)}
\put(81,22){\line(1,0){8}}
\put(90,22){\circle{2}}
\put(95,18){\makebox(0,0){$A_{1}$}}
\put(90,25){\makebox(0,0){$1$}}
\put(95,19){\makebox(0,0)}
\put(91,22){\line(1,0){8}}
\put(100,22){\circle{2}}
\put(100,25){\makebox(0,0){$1$}}
\put(78.5,30.5){$\blacksquare$}
\put(83,26){\makebox(0,0){$A_{0}$}}
\put(80,31){\line(0,-1){8}}
\put(83,32){\makebox(0,0){$1$}}
\put(85,28){\makebox(0,0)}
\end{picture}
\caption{A quiver representation of type $H_{n}$ -- orientation not shown.}
\label{fig-Q-Hn1}
\end{figure}

\item[2.] $\mathbf{a}=(1,1,1,\underbrace{2, 2, \ldots, 2}_{i}, \underbrace{1, \ldots, 1}_{j},1)\in \mathbb{N}^n$ where  $i\geq 1$, $j\geq 0$ with $i+j=n-4$. 

 The corresponding quiver representation in $\textnormal{Rep}(Q ,{\bf a})$ is given in Figure \ref{fig-Q-Hn2}.

\begin{figure}[ht!]
\setlength{\unitlength}{1mm}
\begin{picture}(125,25)(5,10)
\put(30,18){\makebox(0,0){$1$}}
\put(35,25){\makebox(0,0)}
\put(30,22){\circle{2}}
\put(40,22){\circle{2}}
\put(40,18){\makebox(0,0){$1$}}
\put(45,25){\makebox(0,0)}
\put(31,22){\line(1,0){8}}
\put(45,22){\makebox(0,0){$\cdots$}}
{\qbezier(40,15)(45,12)(50,15)}
\put(50,18){\makebox(0,0){$1$}}
\put(45,10){\makebox(0,0){$j$}}
\put(55,25){\makebox(0,0)}
\put(50,22){\circle{2}}
\put(51,22){\line(1,0){8}}
\put(60,18){\makebox(0,0){$2$}}
\put(60,22){\circle{2}}
\put(65,22){\makebox(0,0){$\cdots$}}
{\qbezier(60,15)(70,12)(80,15)}
\put(70,10){\makebox(0,0){$i$}}
\put(70,22){\circle{2}}
\put(70,18){\makebox(0,0){$2$}}
\put(75,25){\makebox(0,0)}
\put(71,22){\line(1,0){8}}
\put(80,22){\circle{2}}
\put(80,18){\makebox(0,0){$2$}}
\put(85,19){\makebox(0,0)}
\put(81,22){\line(1,0){8}}
\put(90,22){\circle{2}}
\put(90,18){\makebox(0,0){$1$}}
\put(95,19){\makebox(0,0)}
\put(91,22){\line(1,0){8}}
\put(100,22){\circle{2}}
\put(100,18){\makebox(0,0){$1$}}
\put(78.5,30.5){$\blacksquare$}
\put(80,31){\line(0,-1){8}}
\put(83,32){\makebox(0,0){$1$}}
\put(100,26){\makebox(0,0){$v_1$}}
\put(85,28){\makebox(0,0)}
\end{picture}
\caption{A quiver representation of type $H_{n}$ with $i\geq 1$, $j\geq 0$ with $i+j=n-4$ -- orientation not shown.}
\label{fig-Q-Hn2}
\end{figure}
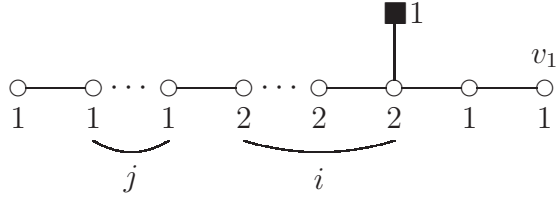

 Notice that a reflection with respect to $v_1$ transforms the representation into a representation as shown in Figure \ref{fig-Q-Dn-1}, whose underlying graph is of the Dynkin type $D_{n-1}$, with the root $\mathbf{a}'=(1,0,1,2,\ldots,2,1,\ldots,1,1)$. Hence the claim follows.\\
\begin{figure}[h!]
\setlength{\unitlength}{1mm}
\begin{picture}(125,25)(0,10)
\put(30,18){\makebox(0,0){$1$}}
\put(35,25){\makebox(0,0)}
\put(30,22){\circle{2}}
\put(40,22){\circle{2}}
\put(40,18){\makebox(0,0){$1$}}
\put(45,25){\makebox(0,0)}
\put(31,22){\line(1,0){8}}
\put(45,22){\makebox(0,0){$\cdots$}}
{\qbezier(40,15)(45,12)(50,15)}
\put(45,10){\makebox(0,0){$j$}}
\put(50,18){\makebox(0,0){$1$}}
\put(55,25){\makebox(0,0)}
\put(50,22){\circle{2}}
\put(51,22){\line(1,0){8}}
\put(60,18){\makebox(0,0){$2$}}
\put(60,22){\circle{2}}
\put(65,22){\makebox(0,0){$\cdots$}}
{\qbezier(60,15)(70,12)(80,15)}
\put(70,10){\makebox(0,0){$i$}}
\put(70,22){\circle{2}}
\put(70,18){\makebox(0,0){$2$}}
\put(75,25){\makebox(0,0)}
\put(71,22){\line(1,0){8}}
\put(80,22){\circle{2}}
\put(80,18){\makebox(0,0){$2$}}
\put(85,19){\makebox(0,0)}
\put(81,22){\line(1,0){8}}
\put(90,22){\circle{2}}
\put(90,18){\makebox(0,0){$1$}}
\put(95,19){\makebox(0,0)}
\put(91,22){\line(1,0){8}}
\put(100,22){\circle{2}}
\put(100,18){\makebox(0,0){$0$}}
\put(78.5,30.5){$\blacksquare$}
\put(80,31){\line(0,-1){8}}
\put(83,32){\makebox(0,0){$1$}}
\put(100,26){\makebox(0,0){$v_1$}}
\put(85,28){\makebox(0,0)}
\end{picture}
\caption{A quiver representation of type $H_{n}$ with $i\geq 1$, $j\geq 0$ and $i+j=n-4$ -- orientation not shown.}
\label{fig-Q-Dn-1}
\end{figure}
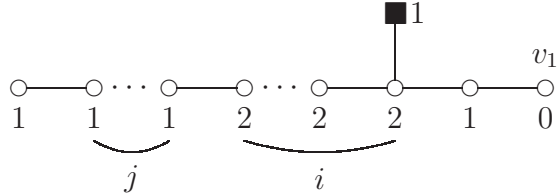

\item[3.] $\mathbf{a}=(1,1,2, \underbrace{2,2, \ldots, 2}_{i}, \underbrace{1, \ldots, 1}_{j},1)$,  where $i\geq 1$, $j\geq 0$ with $i+j=n-4$.

 A generic quiver {   re}presentation in $\textnormal{Rep}(Q ,{\bf a})$ is given in Figure \ref{fig-Q-Hn3}. 
We will show that such representation can be transformed into a quiver representation given in {   Case} 2.
\begin{figure}[h!]
\setlength{\unitlength}{1mm}
\begin{picture}(125,25)(0,10)
\put(30,18){\makebox(0,0){$1$}}
\put(35,25){\makebox(0,0)}
\put(30,22){\circle{2}}
\put(40,22){\circle{2}}
\put(40,18){\makebox(0,0){$1$}}
\put(45,25){\makebox(0,0)}
\put(31,22){\line(1,0){8}}
\put(45,22){\makebox(0,0){$\cdots$}}
\put(50,18){\makebox(0,0){$1$}}
{\qbezier(40,15)(45,12)(50,15)}
\put(45,10){\makebox(0,0){$j$}}
\put(55,25){\makebox(0,0)}
\put(50,22){\circle{2}}
\put(51,22){\line(1,0){8}}
\put(60,18){\makebox(0,0){$2$}}
\put(60,22){\circle{2}}
\put(65,22){\makebox(0,0){$\cdots$}}
{\qbezier(60,15)(70,12)(80,15)}
\put(70,10){\makebox(0,0){$i$}}
\put(70,22){\circle{2}}
\put(70,18){\makebox(0,0){$2$}}
\put(75,25){\makebox(0,0)}
\put(71,22){\line(1,0){8}}
\put(80,22){\circle{2}}
\put(80,18){\makebox(0,0){$2$}}
\put(82,26){\makebox(0,0){$v_3$}}
\put(85,19){\makebox(0,0)}
\put(81,22){\line(1,0){8}}
\put(90,22){\circle{2}}
\put(90,18){\makebox(0,0){$2$}}
\put(90,26){\makebox(0,0){$v_2$}}
\put(95,19){\makebox(0,0)}
\put(91,22){\line(1,0){8}}
\put(100,22){\circle{2}}
\put(100,18){\makebox(0,0){$1$}}
\put(78.5,30.5){$\blacksquare$}
\put(80,31){\line(0,-1){8}}
\put(83,32){\makebox(0,0){$1$}}
\put(100,26){\makebox(0,0){$v_1$}}
\put(85,28){\makebox(0,0)}
\end{picture}
\caption{A quiver representation of type $H_{n}$ with $i\geq 1$, $j\geq 0$ and $i+j=n-4$ -- orientation not shown.}
\label{fig-Q-Hn3}
\end{figure}
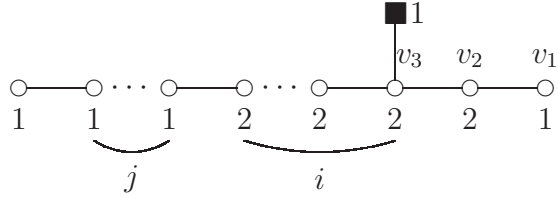

 Let us consider the subquiver $Q'$ consisting of the vertices $v_1,v_2,v_3$. There are four possible choices of orientation for $Q'$.

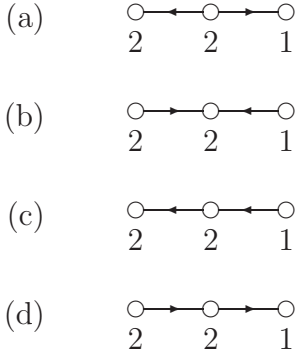
\begin{enumerate}\item[(a)] \setlength{\unitlength}{1mm}
\begin{picture}(80,10)(20,20)
\put(30,18){\makebox(0,0){$2$}}
\put(35,25){\makebox(0,0)}
\put(30,22){\circle{2}}
\put(40,22){\circle{2}}
\put(40,18){\makebox(0,0){$2$}}
\put(45,25){\makebox(0,0)}
\put(31,22){\line(1,0){8}}
\put(37,22){\vector(-1,0){3}}
\put(41,22){\line(1,0){8}}
\put(41,22){\vector(1,0){5}}
\put(50,18){\makebox(0,0){$1$}}
\put(55,25){\makebox(0,0)}
\put(50,22){\circle{2}}
\end{picture}

\item[(b)] \setlength{\unitlength}{1mm}
\begin{picture}(80,10)(20,20)
\put(30,18){\makebox(0,0){$2$}}
\put(35,25){\makebox(0,0)}
\put(30,22){\circle{2}}
\put(40,22){\circle{2}}
\put(40,18){\makebox(0,0){$2$}}
\put(45,25){\makebox(0,0)}
\put(31,22){\line(1,0){8}}
\put(31,22){\vector(1,0){5}}
\put(47,22){\vector(-1,0){3}}
\put(41,22){\line(1,0){8}}
\put(50,18){\makebox(0,0){$1$}}
\put(55,25){\makebox(0,0)}
\put(50,22){\circle{2}}
\end{picture}

\item[(c)]
\setlength{\unitlength}{1mm}
\begin{picture}(80,10)(20,20)
\put(30,18){\makebox(0,0){$2$}}
\put(35,25){\makebox(0,0)}
\put(30,22){\circle{2}}
\put(40,22){\circle{2}}
\put(40,18){\makebox(0,0){$2$}}
\put(45,25){\makebox(0,0)}
\put(31,22){\line(1,0){8}}
\put(41,22){\line(1,0){8}}
\put(37,22){\vector(-1,0){3}}
\put(47,22){\vector(-1,0){3}}
\put(50,18){\makebox(0,0){$1$}}
\put(55,25){\makebox(0,0)}
\put(50,22){\circle{2}}
\end{picture}

\item[(d)]
\setlength{\unitlength}{1mm}
\begin{picture}(80,10)(20,20)
\put(30,18){\makebox(0,0){$2$}}
\put(35,25){\makebox(0,0)}
\put(30,22){\circle{2}}
\put(40,22){\circle{2}}
\put(40,18){\makebox(0,0){$2$}}
\put(45,25){\makebox(0,0)}
\put(31,22){\line(1,0){8}}
\put(41,22){\line(1,0){8}}
\put(31,22){\vector(1,0){5}}
\put(41,22){\vector(1,0){5}}
\put(50,18){\makebox(0,0){$1$}}
\put(55,25){\makebox(0,0)}
\put(50,22){\circle{2}}
\end{picture} \\
\end{enumerate}

 Notice that if the orientation between $v_1,v_2$ and $v_3$ is as in (c) or (d), a reflection with respect to $v_1$ reverses the arrow between $v_1$ and $v_2$ but leaves the dimensions unchanged. Therefore, it transforms the representation into one of the cases (a) or (b). If the orientation between $v_1,v_2$ and $v_3$ is as in (a) or (b), a reflection with respect to $v_2$ transforms the representation into a representation given in Figure \ref{fig-Q-Hn2}. Therefore the claim follows. \\

\item[4.]  $\mathbf{a}=(1,1,2, \underbrace{3, 3, \ldots, 3}_{i},\underbrace{2, \ldots, 2}_{j},1)$, where 
{   $i, j \geq 1$} with $i+j=n-4$.

 Let us consider the representation depicted in Figure \ref{fig-Q-Hn4}. We will show that it can be transformed into a quiver representation studied in {   Case} 3. Let us first study the subquiver consisting of the vertices $v_{n-1},v_{n-2},v_{n-3}$. Notice that there are four possible choices of orientation between $v_{n-1},v_{n-2},v_{n-3}$ similar to the cases (a)-(d) for a root of 
 {   Case 3} above. \\
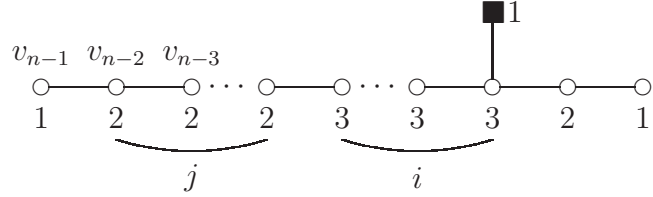
\begin{figure}[h!]
\setlength{\unitlength}{1mm}
\begin{picture}(115,25)(-5,10)
\put(20,18){\makebox(0,0){$1$}}
\put(30,18){\makebox(0,0){$2$}}
\put(35,25){\makebox(0,0)}
\put(20,22){\circle{2}}
\put(21,22){\line(1,0){8}}
\put(30,22){\circle{2}}
\put(40,22){\circle{2}}
\put(20,26){\makebox(0,0){$v_{n-1}$}}
\put(40,18){\makebox(0,0){$2$}}
\put(45,25){\makebox(0,0)}
\put(31,22){\line(1,0){8}}
\put(30,26){\makebox(0,0){$v_{n-2}$}}
\put(40,26){\makebox(0,0){$v_{n-3}$}}
\put(45,22){\makebox(0,0){$\cdots$}}
\put(50,18){\makebox(0,0){$2$}}
{\qbezier(30,15)(40,12)(50,15)}
\put(40,10){\makebox(0,0){$j$}}
\put(55,25){\makebox(0,0)}
\put(50,22){\circle{2}}
\put(51,22){\line(1,0){8}}
\put(60,18){\makebox(0,0){$3$}}
\put(60,22){\circle{2}}
\put(65,22){\makebox(0,0){$\cdots$}}
{\qbezier(60,15)(70,12)(80,15)}
\put(70,10){\makebox(0,0){$i$}}
\put(70,22){\circle{2}}
\put(70,18){\makebox(0,0){$3$}}
\put(75,25){\makebox(0,0)}
\put(71,22){\line(1,0){8}}
\put(80,22){\circle{2}}
\put(80,18){\makebox(0,0){$3$}}
\put(85,19){\makebox(0,0)}
\put(81,22){\line(1,0){8}}
\put(90,22){\circle{2}}
\put(90,18){\makebox(0,0){$2$}}
\put(95,19){\makebox(0,0)}
\put(91,22){\line(1,0){8}}
\put(100,22){\circle{2}}
\put(100,18){\makebox(0,0){$1$}}
\put(78.5,30.5){$\blacksquare$}
\put(80,31){\line(0,-1){8}}
\put(83,32){\makebox(0,0){$1$}}
\put(85,28){\makebox(0,0)}
\end{picture}
\caption{A quiver representation of type $H_{n}$ with 
{   $i, j \geq 1$} and $i+j=n-4$ -- orientation not shown.}
\label{fig-Q-Hn4}
\end{figure}

 Therefore, by a certain number of reflection functors, we can deduce that a representation of $H_n$ with a root as in Figure \ref{fig-Q-Hn4} can be transformed into a representation shown in Figure \ref{fig-Q-Hn6}.
\begin{figure}[h!]
\setlength{\unitlength}{1mm}
\begin{picture}(115,25)(-5,10)
\put(20,18){\makebox(0,0){$0$}}
\put(30,18){\makebox(0,0){$0$}}
\put(35,25){\makebox(0,0)}
\put(20,22){\circle{2}}
\put(30,22){\circle{2}}
\put(40,22){\circle{2}}
\put(40,18){\makebox(0,0){$1$}}
\put(45,25){\makebox(0,0)}
\put(31,22){\line(1,0){8}}
\put(25,22){\makebox(0,0){$\cdots$}}
\put(41,22){\line(1,0){8}}
\put(50,18){\makebox(0,0){$2$}}
\put(55,25){\makebox(0,0)}
\put(50,22){\circle{2}}
\put(51,22){\line(1,0){8}}
\put(60,18){\makebox(0,0){$3$}}
\put(60,22){\circle{2}}
\put(65,22){\makebox(0,0){$\cdots$}}
{\qbezier(60,15)(70,12)(80,15)}
\put(70,10){\makebox(0,0){$i$}}
\put(70,22){\circle{2}}
\put(70,18){\makebox(0,0){$3$}}
\put(75,25){\makebox(0,0)}
\put(71,22){\line(1,0){8}}
\put(80,22){\circle{2}}
\put(80,18){\makebox(0,0){$3$}}
\put(60,26){\makebox(0,0){$v_{i+2}$}}
\put(85,19){\makebox(0,0)}
\put(81,22){\line(1,0){8}}
\put(90,22){\circle{2}}
\put(90,18){\makebox(0,0){$2$}}
\put(70,26){\makebox(0,0){$v_4$}}
\put(78,26){\makebox(0,0){$v_3$}}
\put(95,19){\makebox(0,0)}
\put(91,22){\line(1,0){8}}
\put(100,22){\circle{2}}
\put(100,18){\makebox(0,0){$1$}}
\put(78.5,30.5){$\blacksquare$}
\put(80,31){\line(0,-1){8}}
\put(83,32){\makebox(0,0){$1$}}
\put(85,28){\makebox(0,0)}
\end{picture}
\caption{A quiver representation of type $H_{n}$ with $1\leq i \leq {   n-5}$ -- orientation not shown.}
\label{fig-Q-Hn6}
\end{figure}

 Moreover, another set of reflections with respect to the vertices $v_{i+2}, \ldots, v_4$ in the given order, possibly combined with another set of reflections with respect to the vertices coming before them to reverse arrows if needed,  yields a representation of the quiver given in Figure \ref{fig-Q-Hn7}.

\begin{figure}[h!]
\setlength{\unitlength}{1mm}
\begin{picture}(115,25)(-5,10)
\put(20,18){\makebox(0,0){$0$}}
\put(30,18){\makebox(0,0){$0$}}
\put(35,25){\makebox(0,0)}
\put(20,22){\circle{2}}
\put(30,22){\circle{2}}
\put(40,22){\circle{2}}
\put(40,18){\makebox(0,0){$1$}}
\put(45,25){\makebox(0,0)}
\put(31,22){\line(1,0){8}}
\put(25,22){\makebox(0,0){$\cdots$}}
\put(41,22){\line(1,0){8}}
\put(50,18){\makebox(0,0){$2$}}
\put(55,25){\makebox(0,0)}
\put(50,22){\circle{2}}
\put(51,22){\line(1,0){8}}
\put(60,18){\makebox(0,0){$2$}}
\put(60,22){\circle{2}}
\put(65,22){\makebox(0,0){$\cdots$}}
{\qbezier(60,15)(70,12)(80,15)}
\put(70,10){\makebox(0,0){$i$}}
\put(70,22){\circle{2}}
\put(70,18){\makebox(0,0){$2$}}
\put(75,25){\makebox(0,0)}
\put(71,22){\line(1,0){8}}
\put(80,22){\circle{2}}
\put(80,18){\makebox(0,0){$3$}}
\put(60,26){\makebox(0,0){$v_{i+2}$}}
\put(85,19){\makebox(0,0)}
\put(81,22){\line(1,0){8}}
\put(90,22){\circle{2}}
\put(90,18){\makebox(0,0){$2$}}
\put(70,26){\makebox(0,0){$v_4$}}
\put(78,26){\makebox(0,0){$v_3$}}
\put(91,22){\line(1,0){8}}
\put(100,22){\circle{2}}
\put(100,18){\makebox(0,0){$1$}}
\put(78.5,30.5){$\blacksquare$}
\put(80,31){\line(0,-1){8}}
\put(83,32){\makebox(0,0){$1$}}
\put(76,32){\makebox(0,0){$v_0$}}
\put(100,26){\makebox(0,0){$v_1$}}
\put(90,26){\makebox(0,0){$v_2$}}
\put(85,28){\makebox(0,0)}
\end{picture}
\caption{A quiver representation of type $H_{n}$ with $1\leq i \leq {   n-5}$ -- orientation not shown.}
\label{fig-Q-Hn7}
\end{figure}

 Finally, if $v_3$ is not a sink (or a source), by a series of reflections with respect to other vertices (except $v_0$ so that the dimension at $v_0$ stays $1$), we can transform it in to a sink (or a source). Then a reflection with respect to $v_3$ yields a quiver representation shown in Figure \ref{fig-Q-Hn3}. Therefore, the claim follows by the discussion in {   Case} 3. \\

\item[5.] $(1,1,2, \underbrace{3, 3, \ldots, 3}_{i}, \underbrace{2, \ldots, 2}_{j}, \underbrace{1, \ldots, 1}_{k},1)$ where {   $i,j,k \geq 1$} with $i+j+k=n-4$.
By Lemma \ref{lem-An} and the discussion for {   Case} 4, 
such quiver also defines a linear free divisor. 
\end{enum}

 This concludes the proof of the proposition for $H_n$ whence for all of the series.
\end{proof}

 Since the discriminant in $\textnormal{Rep}(Q ,{\bf a})$ is a linear free divisor, its complement is an open orbit  whose points corresponds to indecomposable representations. Consequently, 

\begin{corollary} 
{    Each root $\bf a$ in the triple root system} of a rational triple quiver is a real Schur root. \end{corollary}

 A linear free divisor in the representation space of a Dynkin quiver is locally quasihomogeneous, that is, it can locally be defined by a quasihomogeneous polynomial with respect to positive weights (\cite[Theorem 3.20]{granger-mond-schulze}). It follows by \cite[Remark 1.7.4]{macarro} that those linear free divisors satisfy the global logarithmic comparison theorem (GLCT). Now we have,

\begin{corollary} All linear free divisors arising from 
{    rational triple quivers with roots in the corresponding triple root systems 
} satisfy the global logarithmic comparison theorem. 
\end{corollary}
\begin{proof} 
{    
The statement follows from 
\cite[Theorem~4.1]{granger-mond-r-s}. }
\end{proof}

\begin{remark} By \cite[Theorem 3.17]{granger-mond-schulze}, all linear free divisors arising from {    rational triple quivers with roots in the corresponding triple root systems 
} are Euler homogeneous, i.e. $\textnormal{Der}(-\textnormal{log } D)$ contains an Euler vector field. 
\end{remark}

\section{Ideas on {   a} generalisation to rational trees with almost reduced Artin cycle}\label{sect-quasi}

{    
In preceding sections, we constructed the root systems and  the linear free divisors corresponding  to a rational triple tree by considering specially  its unique vertex with weight $3$ and  coefficient $1$ in the Artin cycle $Z$ of the tree. 
Here we will consider the rational $m$-tuple trees ${\cal R} $ in which there are multiple vertices with weight $\geq 3$ and  $1$ as their coefficients in the Artin cycle of ${\cal R}$.

\begin{defn} \cite{gustavsen}
For the Artin cycle $Z$ on a rational tree ${\cal R}$, we say that $Z$ is {\it almost reduced} if 
all coefficients of vertices with weight $\ge 3$ in $Z$ are $1$. 
If the Artin cycle $Z$ of a rational tree ${\cal R}$ is almost reduced, then we call 
${\cal R}$ a {\it rational tree with almost reduced Artin cycle}.  
\end{defn}

}

\begin{defn}\label{def-quasidetroot}
{    Let $\cal R $ be an $m$-tuple \rar{} of vertices $E_1, E_2,\ldots ,E_n$.} Consider the set $R({\cal R} )$ of divisors $Y=\sum m_iE_i$ satisfying the following conditions:
\begin{enumerate}
\item  For each $Y$ in $R({\cal R} )$, the support ${\rm Supp}(Y)$  is a connected tree.
\item The coefficients of each $E_i$ in $Y$ are all either positive or negative. 
\item $-2\geq (Y\cdot Y)\geq -m$ where $m=-Z^2$.
\end{enumerate} 
{    In this case, the set  $R({\cal R} )$ is called {\it the almost reduced rational root system}.  As before, we call an element of $R({\cal R})$ a {\it root} and, each $E_i$ a {\it simple (positive) root} of $R({\cal R})$. }  
\end{defn}

It is cumbersome to show all computations for this general case; however, the methods we used in the preceding sections to calculate number of roots and to discuss linear free divisors can be applied directly to {    an \rar}. 
It is crucial to note that {   Lemmas} \ref{lem-suppy} and \ref{support}, which are proved for rational triple root systems,  may not 
be true for every almost reduced rational root system. Hence we added the condition 3 above to Definition \ref{def-quasidetroot}.

{   Now let us consider the case given in Figure \ref{fig-star} which is the most basic form of   
rational trees with almost reduced Artin cycle. Note that $\ell$, the total number of subtrees,  is the valency of the central vertex and $\ell\leq w$ with $w$ weight of the central vertex. 
The highest root is given by the sum of the highest roots of each Dynkin subtree and the coefficient $1$ on the central vertex.  {    In order to compute the multiplicity $m$ of the tree we need to know the contribution of each subtree $\Gamma_i$ to $m$:} The contribution of $\Gamma_i$ of type $D_n$, $E_6$ and $E_7$ to $m$ is  $(-2)$ and, 
type $A_n$, $A_{k,t}$ and $D_q^s$ is $-1$, $-k$ and $r-1$, respectively, when $q=2r$ or $q=2r+1$ (see \cite{m-a-z}). 
Here, $D_q^s$ is a tree of type $D_q$  glued to the central vertex along one of its short tail. 
Thus we get
\begin{eqnarray}-m=Z^2=-w^2 &+&\sum_i\left((-1)\#A_{n_i}+(-2)(\#D_{n_i'}+\#E_6+\#E_7)\right) \nonumber \\
&+&\sum_i\left((-k_i)\#A_{k_i,t_i}+(r_i-1)\#D_{q_i}^{s_i}\right) \end{eqnarray}
where $\#$ denotes the number of trees of the indicated type appearing in  Figure  \ref{fig-star}. 
}

{    
A quiver whose underlying tree is an \rar{} is called an \textit{almost reduced rational  quiver}. }  
The complexity of the study of linear free divisors arising from a $m$-tuple almost reduced rational quivers depends on the number of the Dynkin subquivers. Here we will not attempt to study all cases. However, as we did in Section \ref{sect-free}, we can use a series of reflection functors to relate the divisor with a linear free divisor coming from a Dynkin quiver, and possibly a normal crossing divisor, under some conditions such as  $\langle  \textbf{a},\textbf{a} \rangle=1$. 

{    
\begin{example} Let $Q$ be the almost reduced rational quiver in Figure \ref{fig-quasi-ex-quiver1} which is formed by the quivers of type $A_{4,1}$, $D^s_5$ and $E_6$.  We note that $Z^2=-w_1-w_2-w_3+4$ with $w_1\geq 3$, $w_2\geq 4$ and $w_3\geq 5$ for the underlying rational tree. Then $Q$ with the assigned root determines a linear free divisor. This is shown by using the reflection functors. 
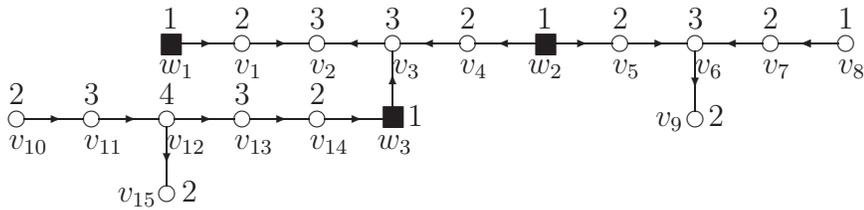
\begin{figure}[h!]
\setlength{\unitlength}{1mm}
\begin{picture}(120,37)(5,-2)
\put(30.5,24){$1$}%
\put(30,20.5){$\blacksquare$}
\put(30,18){$w_1$}%
\put(32,22){\vector(1,0){5}}
\put(32,22){\line(1,0){8}}
\put(41,22){\circle{2}}
\put(42,22){\line(1,0){8}}
\put(42,22){\vector(1,0){5}}
\put(40,18){$v_1$}
\put(40,24){$2$}%
\put(51,22){\circle{2}}
\put(52,22){\line(1,0){8}}
\put(60,22){\vector(-1,0){5}}
\put(50,18){$v_2$}
\put(50,24){$3$}%
\put(61,22){\circle{2}}
\put(62,22){\line(1,0){8}}
\put(61,18){$v_3$}
\put(60,24){$3$}%
\put(71,22){\circle{2}}
\put(72,22){\line(1,0){8}}
\put(70,22){\vector(-1,0){5}}
\put(70,18){$v_4$}
\put(70,24){$2$}%
\put(80.25,24){$1$}%
\put(79.75,20.5){$\blacksquare$}
\put(80,22){\vector(-1,0){5}}
\put(82,22){\vector(1,0){5}}
\put(79,18){$w_2$}%
\put(82.25,22){\line(1,0){8}}
\put(91.25,22){\circle{2}}
\put(92.25,22){\line(1,0){8}}
\put(92.25,22){\vector(1,0){5}}
\put(90.25,18){$v_5$}
\put(90.25,24){$2$}%
\put(101.25,22){\circle{2}}
\put(102.25,22){\line(1,0){8}}
\put(110.25,22){\vector(-1,0){5}}
\put(101.25,18){$v_6$}
\put(100.25,24){$3$}%
\put(111.25,22){\circle{2}}
\put(112.25,22){\line(1,0){8}}
\put(120.25,22){\vector(-1,0){5}}
\put(110.25,18){$v_7$}
\put(110.25,24){$2$}%
\put(121.25,22){\circle{2}}
\put(101.25,21){\line(0,-1){8}}
\put(120.25,22){\vector(-1,0){5}}
\put(120.25,18){$v_8$}
\put(120.25,24){$1$}%
\put(101.25,12){\circle{2}}
\put(101.25,21){\vector(0,-1){5}}
\put(96,11){$v_9$}
\put(103,11){$2$}%
\put(63,11){$1$}%
\put(59.5,10.75){$\blacksquare$}
\put(59,8){$w_3$}%
\put(61,21){\line(0,-1){8}}
\put(61,13){\vector(0,1){5}}
\put(11,12){\circle{2}}
\put(12,12){\line(1,0){8}}
\put(12,12){\vector(1,0){5}}
\put(10,8){$v_{10}$}
\put(10,14){$2$}%
\put(21,12){\circle{2}}
\put(22,12){\line(1,0){8}}
\put(22,12){\vector(1,0){5}}
\put(20,8){$v_{11}$}
\put(20,14){$3$}%
\put(31,12){\circle{2}}
\put(32,12){\line(1,0){8}}
\put(32,12){\vector(1,0){5}}
\put(31,8){$v_{12}$}
\put(30,14){$4$}%
\put(41,12){\circle{2}}
\put(42,12){\line(1,0){8}}
\put(42,12){\vector(1,0){5}}
\put(40,8){$v_{13}$}
\put(40,14){$3$}%
\put(51,12){\circle{2}}
\put(52,12){\line(1,0){8}}
\put(52,12){\vector(1,0){5}}
\put(50,8){$v_{14}$}
\put(50,14){$2$}%
\put(31,11){\line(0,-1){8}}
\put(31,2){\circle{2}}
\put(31,11){\vector(0,-1){5}}
\put(24.5,1){$v_{15}$}
\put(33,1){$2$}%
\end{picture}
\caption{An almost reduced rational quiver with its highest root.}
\label{fig-quasi-ex-quiver1}
\end{figure}
First, we apply reflection functors with respect to the vertices
\begin{eqnarray*}&& v_2,v_3,v_1,v_2,v_4,w_1,v_1,v_9,v_6,v_5,\\ && v_7,v_9,v_6,v_{10},
v_{11},v_{12},v_{13},v_{14},v_{15},\\&& v_{10},v_{11},v_{12},
v_{10},v_{11},v_{15},v_{13},v_{12}\end{eqnarray*}
in the given order, to transform $Q$  into the quiver $Q'$ shown in Figure \ref{fig-quasi-ex-quiver2}.  
\begin{figure}[h!]
\setlength{\unitlength}{1mm}
\begin{picture}(120,37)(5,-2)
\put(30.5,24){$0$}%
\put(30,20.5){$\blacksquare$}
\put(30,18){$w_1$}%
\put(32,22){\line(1,0){8}}
\put(41,22){\circle{2}}
\put(42,22){\line(1,0){8}}
\put(40,18){$v_1$}
\put(40,24){$0$}%
\put(51,22){\circle{2}}
\put(52,22){\line(1,0){8}}
\put(52,22){\vector(1,0){5}}
\put(50,18){$v_2$}
\put(50,24){$1$}%
\put(61,22){\circle{2}}
\put(62,22){\line(1,0){8}}
\put(70,22){\vector(-1,0){5}}
\put(61,18){$v_3$}
\put(60,24){$2$}%
\put(71,22){\circle{2}}
\put(72,22){\line(1,0){8}}
\put(72,22){\vector(1,0){5}}
\put(70,18){$v_4$}
\put(70,24){$1$}%
\put(80.25,24){$1$}%
\put(79.75,20.5){$\blacksquare$}
\put(79,18){$w_2$}%
\put(82.25,22){\line(1,0){8}}
\put(91.25,22){\circle{2}}
\put(92.25,22){\line(1,0){8}}
\put(100.25,22){\vector(-1,0){5}}
\put(90.25,22){\vector(-1,0){5}}
\put(90.25,18){$v_5$}
\put(90.25,24){$1$}%
\put(101.25,22){\circle{2}}
\put(102.25,22){\line(1,0){8}}
\put(102.25,22){\vector(1,0){5}}
\put(101.25,18){$v_6$}
\put(100.25,24){$1$}%
\put(111.25,22){\circle{2}}
\put(112.25,22){\line(1,0){8}}
\put(112.25,22){\vector(1,0){5}}
\put(110.25,18){$v_7$}
\put(110.25,24){$1$}%
\put(121.25,22){\circle{2}}
\put(101.25,21){\line(0,-1){8}}
\put(120.25,18){$v_8$}
\put(120.25,24){$1$}%
\put(101.25,12){\circle{2}}
\put(101.25,21){\vector(0,-1){5}}
\put(96,11){$v_9$}
\put(103,11){$1$}%
\put(63,11){$1$}%
\put(59.65,10.5){$\blacksquare$}
\put(59,8){$w_3$}%
\put(61,21){\line(0,-1){8}}
\put(61,21){\vector(0,-1){5}}
\put(11,12){\circle{2}}
\put(12,12){\line(1,0){8}}
\put(12,12){\vector(1,0){5}}
\put(10,8){$v_{10}$}
\put(10,14){$1$}%
\put(21,12){\circle{2}}
\put(22,12){\line(1,0){8}}
\put(22,12){\vector(1,0){5}}
\put(20,8){$v_{11}$}
\put(20,14){$1$}%
\put(31,12){\circle{2}}
\put(32,12){\line(1,0){8}}
\put(40,12){\vector(-1,0){5}}
\put(31,8){$v_{12}$}
\put(30,14){$1$}%
\put(41,12){\circle{2}}
\put(42,12){\line(1,0){8}}
\put(50,12){\vector(-1,0){5}}
\put(40,8){$v_{13}$}
\put(40,14){$1$}%
\put(51,12){\circle{2}}
\put(52,12){\line(1,0){8}}
\put(60,12){\vector(-1,0){5}}
\put(50,8){$v_{14}$}
\put(50,14){$1$}%
\put(31,11){\line(0,-1){8}}
\put(31,2){\circle{2}}
\put(31,3){\vector(0,1){5}}
\put(24.5,1){$v_{15}$}
\put(33,1){$1$}%
\end{picture}
\caption{The quiver $Q'$ corresponding to $Q$ in Figure \ref{fig-quasi-ex-quiver1}.}
\label{fig-quasi-ex-quiver2}
\end{figure}
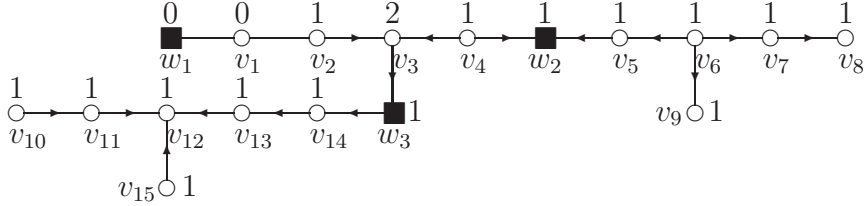

Next, we apply reflection functors with respect to the vertices
\[ v_8,v_7,v_9,v_6,v_5,w_2,v_{10},v_{11},v_{15},v_{12}, v_{13}, v_{14}  \] 
in the given order, to transform $Q'$  into the quiver $Q''$ shown in Figure \ref{fig-quasi-ex-quiver3}.

\begin{figure}[h!]
\setlength{\unitlength}{1mm}
\begin{picture}(120,37)(5,-2)
\put(30.5,24){$0$}%
\put(30,20.5){$\blacksquare$}
\put(30,18){$w_1$}%
\put(32,22){\line(1,0){8}}
\put(41,22){\circle{2}}
\put(42,22){\line(1,0){8}}
\put(40,18){$v_1$}
\put(40,24){$0$}%
\put(51,22){\circle{2}}
\put(52,22){\line(1,0){8}}
\put(52,22){\vector(1,0){5}}
\put(50,18){$v_2$}
\put(50,24){$1$}%
\put(61,22){\circle{2}}
\put(62,22){\line(1,0){8}}
\put(70,22){\vector(-1,0){5}}
\put(61,18){$v_3$}
\put(60,24){$2$}%
\put(71,22){\circle{2}}
\put(72,22){\line(1,0){8}}
\put(70,18){$v_4$}
\put(70,24){$1$}%
\put(80.25,24){$0$}%
\put(79.75,20.5){$\blacksquare$}
\put(79,18){$w_2$}%
\put(82.25,22){\line(1,0){8}}
\put(91.25,22){\circle{2}}
\put(92.25,22){\line(1,0){8}}
\put(90.25,18){$v_5$}
\put(90.25,24){$0$}%
\put(101.25,22){\circle{2}}
\put(102.25,22){\line(1,0){8}}
\put(101.25,18){$v_6$}
\put(100.25,24){$0$}%
\put(111.25,22){\circle{2}}
\put(112.25,22){\line(1,0){8}}
\put(110.25,18){$v_7$}
\put(110.25,24){$0$}%
\put(121.25,22){\circle{2}}
\put(101.25,21){\line(0,-1){8}}
\put(120.25,18){$v_8$}
\put(120.25,24){$0$}%
\put(101.25,12){\circle{2}}
\put(96,11){$v_9$}
\put(103,11){$0$}%
\put(63,11){$1$}%
\put(59.65,10.5){$\blacksquare$}
\put(59,8){$w_3$}%
\put(61,21){\line(0,-1){8}}
\put(61,21){\vector(0,-1){5}}
\put(11,12){\circle{2}}
\put(12,12){\line(1,0){8}}
\put(10,8){$v_{10}$}
\put(10,14){$0$}%
\put(21,12){\circle{2}}
\put(22,12){\line(1,0){8}}
\put(20,8){$v_{11}$}
\put(20,14){$0$}%
\put(31,12){\circle{2}}
\put(32,12){\line(1,0){8}}
\put(31,8){$v_{12}$}
\put(30,14){$0$}%
\put(41,12){\circle{2}}
\put(42,12){\line(1,0){8}}
\put(40,8){$v_{13}$}
\put(40,14){$0$}%
\put(51,12){\circle{2}}
\put(52,12){\line(1,0){8}}
\put(50,8){$v_{14}$}
\put(50,14){$0$}%
\put(31,11){\line(0,-1){8}}
\put(31,2){\circle{2}}
\put(24.5,1){$v_{15}$}
\put(33,1){$0$}%
\end{picture}
\caption{The quiver $Q''$ corresponding to $Q'$ in Figure \ref{fig-quasi-ex-quiver2}.}
\label{fig-quasi-ex-quiver3}
\end{figure}
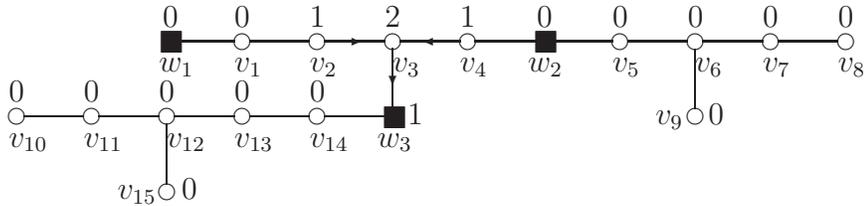

Notice that the subquiver of $Q''$ formed by the vertices $v_2,v_3,v_4$ and $w_3$ is of Dynkin quiver type $D_4$ with the root $(1,2,1,1)$.  So, by the results of \cite{granger-mond-schulze}, 
the quiver $Q$ also defines a linear free divisor.
\end{example}
} 

Therefore, we are motivated to state the following conjecture.
\begin{conj} If $Q$ is an almost reduced rational quiver and  $\mathbf{a}$ is a root in $R(Q)$ then $(Q,\mathbf{a})$ defines a linear free divisor. \end{conj}

\section*{Acknowledgments}
The third author is supported by the project 113F293 under the program 1001 of the Scientific and Technological Research Council of Turkey. She also thanks the Department of Mathematics at Columbia University for their hospitality during this work. 
{   
The first author was partially supported by 
JSPS KAKENHI Grant Number JP23540044 and JP15K04814.} 
The authors would also like to thank the referee for her/his careful reading and constructive comments.


   \Addresses


\begin{thebibliography}{99}

\bibitem{artin} 
M. Artin, On isolated rational singularities of surfaces, {\em Amer. J. Math.} \textbf{88} (1966), 129--136.

\bibitem{bern-gel-pono}
I. N. Bernstein, I. M. Gel'fand and V. A. Ponomarev, Coxeter functors and Gabriel's theorem, {\em Uspehi Mat. Nauk} \textbf{28} (1973) 19--33, translated in Russion Math. Surveys \textbf{28} (1973), 17--32.


\bibitem{bourbaki} 
N. Bourbaki, Groupes et alg\'ebres de Lie, Ch. IV,V,VI, Hermann, Paris, 1968.


\bibitem{buch-mond}
R.-O. Buchweitz and D. Mond, Linear free divisors and quiver representations, Singularities and computer algebra, {\em London Math. Soc. Lecture Note Ser.} \textbf{324}, Cambridge Univ. Press, Cambridge,  (2006), 41--77.

\bibitem{dejong2}
T. de Jong, Determinantal rational surface singularities, {\em Comp. Math.} \textbf{113} (1998), 67--90.

\bibitem{gabriel-2}
P. Gabriel, R{\'e}presentation ind{\'e}composables, S{\'e}minaire Bourbaki, Expos{\'e} \textbf{444} (1974), in {\em Springer Lecture Notes} 431 (1975), 143--169.

\bibitem{granger-mond-schulze}
M. Granger, D. Mond and M. Schulze, Free divisors in prehomogeneous vector spaces, {\em Proc. London Math. Soc.} \textbf{102} (5) (2011), 923--950.

\bibitem{granger-mond-r-s}
M. Granger, D. Mond, Nieto-Reyes and M. Schulze, Linear free divisors and the global logarithmic comparison theorem, {\em Ann. Inst. Fourier (Grenoble)} \textbf{59} (2) (2009), 811--850.

\bibitem{grauert}
H. Grauert, {\" U}ber Modifikationen und exzeptionnelle analytische Mengen, {\em Math. Ann. } \textbf{146} (1962), 331--368.

\bibitem{gustavsen}
T. S. Gustavsen, On the cotangent cohomology of rational surface singularities with almost reduced Artin cycle, \textit{Math. Nachr.} \textbf{279} (11)  (2006), 1185--1194.

\bibitem{kac}
V. Kac, Infinite root systems, representations of graphs and invariant theory, {\em Invent. Math.} \textbf{56} (1980), 57--92.

\bibitem{sato-kimura}
T. Kimura and M. Sato, A classification of irreducible prehomogeneous vector spaces and their relative invariants, {\em Nagaya Math. J.} \textbf{65} (1977), 1--155.

\bibitem{laufer}
H. Laufer, {\em Normal two-dimensional singularities}, Annals of Math. Studies 71, Princeton University Press, 1971.

\bibitem{Le-tosun} 
D.T. L{\^e} and M. Tosun, Combinatorics of rational surface singularities, {\em Comment. Math. Helv.} \textbf{79} (2004), 582--604.

\bibitem{macarro}
L. Narv{\'a}ez Macarro, Linearity conditions on the Jacobian ideal and logarithmic-meromorphic comparison for free divisors, {\em Singularities I, Contemp. Math,} Amer. Math. Soc., Providence, RI, \textbf{474} (2008), 245--269.


\bibitem{m-a-z}
Z. Oer, A. Ozkan and M. Tosun, Classification of rational singularities of multiplicity 5, {\em Int. J. of Pure and Appl. Mathematics} \textbf{41} (1) (2007), 85--110.

\bibitem{pinkham}
H. Pinkham, {\em Singularit{\'e}s rationalles de surfaces}, S{\'e}minaire sur les singularit{\'e}s des surfaces, Lecture Notes in Mathematics, vol. 777, Springer-Verlag, 1980.
  
\bibitem{riemens}
O. Riemenschneider, Die invarianten der endlichen Untergruppen von $GL(2,\mathbb{C})$, {\em Math. Z.} \textbf{153} (1977), 37--50.

\bibitem{saito}
K. Saito, Theory of logarithmic differential forms and logarithmic vector fields, {\em J. Fac. Sci. Univ. Tokyo Sect. Math.} \textbf{27} (1980), 265--291.

\bibitem{spivakovsky}
M. Spivakovsky, Sandwiched singularities and desingularization of surfaces by normalized Nash transformations, {\em Ann. of Math.} \textbf{131} (1990), 411--491.

\bibitem{M} 
M. Tosun, ADE surface singularities, chambers and toric varieties, {\em Proceeding of Franco-Japanese CIRM, Soc. Math. de France, S{\' e}minaire et Congr{\' e}s} \textbf{10} (2005), 341--350.

\bibitem{tjurina} 
G.N. Tyurina, Absolute isolatedness of rational singularities and rational triple points, {\em Fonc. Anal. Appl.} \textbf{2} (1968), 324--332.

\bibitem{wahl-equations}
J. M. Wahl, Equations defining rational singularities, {\em Ann. Scient. Ec. Norm. Sup. $4^\circ$ serie}, \textbf{10} (2) (1977), 231--264.

\bibitem{wunram}
J. Wunram, Reflexive modules on quotient surface singularities, {\em Math. Ann. } \textbf{279} (4) (1988), 83--598.


\end{thebibliography}
\end{document}